\documentclass{amsart}
\usepackage{hyperref}
\usepackage{amsfonts}
\usepackage{amsmath}
\usepackage{amssymb}
\usepackage{amscd}
\usepackage{graphics}
\usepackage{latexsym}
\usepackage{color}
\usepackage{epsfig}
\usepackage{amsopn}

\newtheorem{theorem}{Theorem}[section]
\newtheorem{proposition}[theorem]{Proposition}
\newtheorem{lemma}[theorem]{Lemma}
\newtheorem{corollary}[theorem]{Corollary}

\theoremstyle{remark}
\newtheorem{remark}[theorem]{Remark}
\newtheorem{notation}[theorem]{Notation}
\newtheorem{algorithm}[theorem]{Algorithm}

\newtheorem{example}[theorem]{Example}

\newtheorem{definition}[theorem]{Definition}

\newcommand{\QP}{P} 

\newcommand{\QQ}{\mathbb{Q}}
\newcommand{\CC}{\mathbb{C}}
\newcommand{\OO}{\mathcal{O}}

\newcommand{\ZZ}{\mathbb{Z}}

\newcommand{\chs}[2]{
\left(\begin{smallmatrix} #1 \\#2 
        \end{smallmatrix} \right)}

\begin{document}
\Large

\title{Singularities of generalized Richardson varieties}

\author{Sara Billey}
\address{University of Washington, Department of Mathematics, Seattle, WA 98195}
\email{billey@math.washington.edu}
\author{Izzet Coskun} 

\date{\today} 

\address{University of Illinois at
  Chicago, Department of Mathematics, Statistics and Computer Science, Chicago, IL 60607}

\email{coskun@math.uic.edu}

  
\thanks{During the preparation of this article the first author was
supported by the National Science Foundation (NSF) grant DMS-0800978,
and the second author was partially supported by the NSF grant
DMS-0737581, NSF CAREER grant DMS-0950951535 and an Arthur P. Sloan Foundation
Fellowship.}

\begin{abstract}

Richardson varieties play an important role in intersection theory and
in the geometric interpretation of the Littlewood-Richardson Rule for
flag varieties.  We discuss three natural generalizations of
Richardson varieties which we call projection varieties, intersection
varieties, and rank varieties.  In many ways, these varieties are more
fundamental than Richardson varieties and are more easily amenable to
inductive geometric constructions. In this paper, we study the
singularities of each type of generalization. Like Richardson
varieties, projection varieties are normal with rational
singularities. We also study in detail the singular loci of projection
varieties in Type A Grassmannians.  We use Kleiman's Transversality
Theorem to determine the singular locus of any intersection variety in
terms of the singular loci of Schubert varieties.  This is a
generalization of a criterion for any Richardson variety to be smooth
in terms of the nonvanishing of certain cohomology classes which has
been known by some experts in the field, but we don't believe has been
published previously.
\end{abstract}

\maketitle \tableofcontents

\section{Introduction}\label{s:intro}

A {\em Richardson variety} is the intersection of two Schubert
varieties in general position in a homogeneous variety $G/P$. Their
cohomology classes encode information that plays a significant role in
algebraic geometry, representation theory and combinatorics
\cite{Fulton,Fulton-book,kumar-book,manivel-book}.  In recent years,
the study of the singularities of Richardson varieties has received a
lot of interest. We refer the reader to \cite{Brion:flag} for general
results about the singularities of Richardson varieties and to
\cite{kreiman} for a detailed study of the singularities of Richardson
varieties in Type A Grassmannians.

In this paper, we study three natural generalizations of Richardson
varieties called intersection varieties, projection varieties and rank
varieties.  We extend several of the results of \cite{kreiman}
pertaining to smoothness criteria, singular loci and multiplicities to
these varieties in $G/P$ for arbitrary semi-simple algebraic groups
$G$ and parabolic subgroups $P$. However, it is important to note that
while in \cite{kreiman} the authors work over algebraically closed
fields of arbitrary characteristic, we require the ground field to
have characteristic zero. \smallskip

The first generalization of Richardson varieties that we discuss is
the \textit{intersection varieties}.  These varieties are simply the
intersection of any finite number of general translates of Schubert
varieties and they appear throughout the literature on Schubert
calculus.  We recall how Kleiman's Transversality Theorem
\cite{kleiman} determines the singular locus of any intersection
variety in terms of the singular loci of Schubert varieties. In
Corollary~\ref{richardson}, we characterize the smooth Richardson
varieties in terms of vanishing conditions on certain products of
cohomology classes for Schubert varieties.  As an application, we show
that a Richardson variety in the Grassmannian variety $G(k,n)$ is
smooth if and only if it is a Segre product of Grassmannians (see
Corollary~\ref{segre}).

The second  generalization of Richardson varieties that we discuss is
the \textit{projection varieties}.
Given $G/P$ as above, let $Q\subset G$ be another parabolic subgroup
containing $P$.  Thus, we have the natural projection
\[
\pi_{Q} : G/P \rightarrow G/Q.   
\]
A {\em projection variety} is the image of a Richardson variety
$R(u,v)$ under a projection $\pi_{Q}$ with its reduced induced
structure.  Projection varieties naturally arise in inductive
constructions such as the Bott-Samelson resolutions.  For example,
they are related to the stratifications used by Lusztig, Postnikov and
Rietsch in the theory of total positivity
\cite{Lusztig94,postnikov-2006,rietsch-2006} and Brown-Goodearl-Yakimov \cite{BGY} in Poisson geometry. 
They generalize the (closed) positroid varieties defined by Knutson, Lam and
Speyer in \cite[Section 5.4]{KLS} and
they play a crucial role in the positive geometric
Littlewood-Richardson rule for Type A flag varieties in
\cite{coskun:LR}.  Since the set of projection varieties is closed
under the projection maps among flag varieties, projection varieties
form a more fundamental class of varieties than Richardson varieties.
Our first theorem about the singularities of projection varieties is
the following, generalizing \cite[Cor. 7.9 and Cor. 7.10]{KLS}.

\begin{theorem}\label{main-1}
Let $G$ be a complex simply connected algebraic group and $Q$ be a
parabolic subgroup of $G$. Then all projection varieties in $G/Q$ are
normal and have rational singularities.
\end{theorem}   

\smallskip

In fact, we will prove a much more general statement (Theorem
\ref{Mori}) about the restriction of Mori contractions to subvarieties
satisfying certain cohomological properties. Since Richardson
varieties satisfy these cohomological properties, it will follow that
projections of Richardson varieties have rational singularities
proving the theorem.  By \cite[Theorem 3]{kovacs}, this implies that
projection varieties are Cohen-Macaulay.  

\smallskip

To define the third family of varieties related to Richardson
varieties, we specialize to Type A ($G=GL(n)$) Grassmannian projection
varieties. In this case, we consider $G/P$ to be the partial flag
variety $Fl(k_1, \dots, k_m; n)$ consisting of partial flags $$V_1
\subset \cdots \subset V_m,$$ where each $V_i$ is a complex vector
space of dimension $k_i$.  Let $G(k,n)$ be the Grassmannian variety of
$k$-dimensional subspaces in an $n$-dimensional complex vector space
$V$.  We can realize $G(k,n)$ as $G/Q$ where $Q$ is a maximal
parabolic subgroup of $G$.  Let $$\pi: Fl(k_1, \dots, k_m; n)
\longrightarrow G(k_m,n)$$ denote the natural projection morphism
defined by $\pi(V_1, \dots, V_m) = V_m$.  \smallskip

A {\em Grassmannian projection variety} in $G(k_m,n)$ is the image
$\pi (R(u,v))$ of a Richardson variety $R(u,v) \subset Fl(k_1, \dots,
k_m; n)$ with its reduced induced structure.  It is convenient to have
a characterization of Grassmannian projection varieties without
referring to the projection of a particular Richardson variety. We
introduce rank sets and rank varieties to obtain such a
characterization.

Fix an ordered basis $e_1, \dots, e_n$ of $V$.  If $W$ is a vector
space spanned by a consecutive set of basis elements $e_i, e_{i+1},
\dots, e_j$, let $l(W)= i$ and $r(W)= j$. A {\em rank set} $M$ for
$G(k,n)$ is a set of $k$ vector spaces $M=\{ W_1, \dots, W_k \}$,
where each vector space is the span of (non-empty) consecutive
sequences of basis elements and $l(W_i) \not= l(W_j)$ and $r(W_i)
\not= r(W_j)$ for $i \not= j$. Observe that the number of vector
spaces $k$ is equal to the dimension of subspaces parameterized by
$G(k,n)$. Two rank sets $M_1$ and $M_2$ are equivalent if they are
defined with respect to the same ordered basis of $V$ and consist of
the same set of vector spaces.

Given a rank set $M$, we can define an irreducible subvariety $X(M)$
of $G(k,n)$ associated to $M$ as follows.  The \textit{rank variety}
$X(M)$ is the subvariety of $G(k,n)$ defined by the Zariski closure of
the set of $k$-planes in $V$ that have a basis $b_1, \dots, b_k$ such
that $b_i \in W_i$ for $W_i \in M$.  For example, $G(k,n)$ is itself a
rank variety corresponding with the rank set $M=\{W_{1},\ldots , W_{k}
\}$ where each $W_{i}=<e_{i},\dotsc, e_{n-k+i}>$.   
 
In Theorem~\ref{projection=rank}, we prove that $X \subset G(k,n)$ is
a projection variety if and only if $X$ is a rank variety.  In
particular, the Richardson varieties in $G(k,n)$ are rank varieties.
The singular loci of rank varieties or equivalently of Grassmannian
projection varieties can be characterized as follows.

\begin{theorem}\label{main-2}
Let $X$ be a rank variety in $G(k_m,n)$. 
\begin{enumerate}
\item There
exists a partial flag variety $F(k_1, \dots, k_m; n)$ and a Richardson
variety $$R(u,v) \subset F(k_1, \dots, k_m; n)$$ such that
$$\pi|_{R(u,v)}: R(u,v) \longrightarrow X$$ is a birational morphism
onto $X$.
\item The singular locus of $X$ is the set of points $x \in X$
such that either $\pi^{-1}|_{R(u,v)}(x) \in R(u,v)$ is singular or
$\pi^{-1}|_{R(u,v)} (x)$ is positive dimensional: $$X^{sing} = \{ x
\in X \ | \ \dim(\pi^{-1}|_{R(u,v)}(x)) >0 \ \mbox{or} \
\pi^{-1}|_{R(u,v)}(x) \in R(u,v)^{sing}\}.$$
\item  The
singular locus of a projection variety is an explicitly determined
union of projection varieties.
\end{enumerate}
\end{theorem}

In Lemma~\ref{dim}, we give a simple formula for the dimension of a
rank variety.  In Corollary~\ref{c:stirling}, we relate the
enumeration of rank varieties by dimension to a $q$-analog of the
Stirling numbers of the second kind.

\smallskip

The organization of this paper is as follows. In
Section~\ref{s:intersection.varieites}, we will introduce our notation
and study the singular loci of Richardson varieties and other
intersection varieties in homogeneous varieties. We recall Kleiman's
Transversality Theorem.  This allows us to completely characterize the
singular loci of intersection varieties in terms of the singular loci
of Schubert varieties in Proposition~\ref{intersection}.  An example
in the Grassmannian $G(3,8)$ is given showing that Richardson
varieties can be singular at every $T$-fixed point.  As corollaries,
we discuss some special properties of intersection varieties in
Grassmannians.  In Section~\ref{s:main-1}, we will prove a general
theorem about the singularities of the image of a subvariety
satisfying certain cohomological properties under a Mori
contraction. This will immediately imply Theorem~\ref{main-1}. In
Section~\ref{s:grassmannian.projetion.varieties}, we will undertake a
detailed analysis of the singularities of Grassmannian projection
varieties and rank varieties.  In particular, the proof of
Theorem~\ref{main-2} follows directly from Corollary~\ref{important}.

\smallskip

\noindent {\bf Acknowledgments:} We would like to thank Dave Anderson,
Lawrence Ein, S\'{a}ndor Kov\'{a}cs, Max Lieblich, Stephen Mitchell, and Lauren Williams for
helpful and stimulating discussions. We are grateful to the American
Mathematics Institute where this project was started, for providing a
stimulating work environment, and to the organizers of the
Localization Techniques in Equivariant Cohomology Workshop, namely
William Fulton, Rebecca Goldin, and Julianna Tymoczko.

\section{The singularities of intersection varieties and Richardson varieties}\label{s:intersection.varieites}

In this section, we review the necessary notation and background for
this article.  In particular, we recall Kleiman's Transversality
Theorem and review its application to the singular loci of Richardson
varieties and intersection varieties.  For the convenience of the reader, we included the proofs of results such as Theorem \ref{thm:kleiman.variation} and Corollary \ref{richardson} when our formulation differed from what is commonly available in the literature.  For further background, we
recommend
\cite{Brion:flag,fulton-harris,GL-book,harris,Hum-LAG,wiki:flags}.

Given a projective variety $X$ and a point $p \in X$, let $T_p X$
denote the Zariski tangent space to $X$ at $p$. Then, $p$ is a
\textit{singular point} of $X$ if $\dim T_{p} X > \dim X$, and $p$ is
\textit{smooth} if $\dim T_{p} X = \dim X$.  Let $X^{sing}$ denote the
set of all singular points in $X$.

Let $G$ denote a simply connected, semi-simple algebraic group over
the complex numbers $\CC$. Fix a maximal torus $T$ and a Borel
subgroup $B$ containing $T$. Let $\QP$ denote a parabolic subgroup of
$G$ containing $B$. Let $W=N(T)/T$ denote the Weyl group of $G$, and
let $W_{\QP}$ denote the Weyl group of $\QP$.  We will abuse notation
by considering any element $u \in W$ to also represent a choice of
element in the coset $uT \subset G$.  In particular, we consider $W
\subset G$ via this choice.  Let $e_{u}=u\QP$, since $T \subset P$
this point is well defined in $G/P$.  The points $\{e_u : u\in
W/W_{\QP}\}$ are the $T$-fixed points in $G/\QP$.

For an element $u \in W/W_{\QP}$, the {\em Schubert variety} $X_u$ is
the Zariski closure of the $B$-orbit of $e_u = u\QP$ in $G/\QP $.
Thus, $X_{u}$ is the union of $B$-orbits $Be_{t}$ for $t\leq u$ in the
Bruhat order on $W/W_{\QP}$.  Since we are working over $\mathbb{C}$,
the (complex) dimension of $X_{u}$ is  the length of $u$ as an element of
$W/W_{\QP}$.

The singular locus of $X_{u}$ is also a $B$-stable subvariety of
$G/\QP$, hence it is a union of Schubert varieties. The typical way of
studying the singularities of Schubert varieties is in terms of the
$T$-fixed points.  In particular, $p \in X_{u}$ is a smooth point if
and only if there exists a $t\in W/W_{\QP}$ such that $p \in Be_{t}$ and
$e_{t}$ is a smooth point of $X_{u}$.  There are many effective tools
for determining if $e_{t}$ is a smooth point in $X_{u}$ and exactly
which elements of $W/W_{\QP}$ index the Schubert varieties which form
the irreducible components of the singular locus of $X_{u}$, see
\cite{BLak,BW-sing,CK,cortez,klr,kumar,manivel,polo94}.

Let  $w_{0}$ be the unique longest element in
$W$. For $v \in W/W_{\QP}$, define the {\em
opposite Schubert variety}, denoted $X^{v}$, as the Schubert variety $w_0 X_v$. We caution the reader that some authors use $X^v$  to denote the Schubert variety in the Poincar\'{e} dual class. The {\em Richardson variety} $R(u,v) \in G/\QP $ is defined as the
intersection of the two Schubert varieties $X_u$ and $X^v$.  $R(u,v)$
is empty unless $u\geq w_0 v$ in the Bruhat order, in which case the dimension
of $R(u,v)$ is $l(u)-l(w_0 v)$.

For much of the discussion, there is no reason to restrict to
Richardson varieties. Intersection varieties provide a more natural
set of varieties to consider.  Let $gX_{u}$ denote the translate of
the Schubert variety $X_{u}$ under the action of $g \in G$ by left
multiplication on $G/\QP$.  
\smallskip

\begin{definition}
Let $u_1, \dots, u_r \in W/W_{\QP}$, and let $g_{\bullet} = (g_1,
\dots, g_r)$ be a general $r$-tuple of elements in $G^r$.  The {\em
intersection variety} $R(u_1, \dots, u_r; g_{\bullet})$ is defined as
the intersection of the translated Schubert varieties $g_{i}X_{u_i}$ in $G/\QP$:
$$R(u_1, \dots, u_r; g_{\bullet})= g_{1}X_{u_1} \cap \cdots \cap
g_{r}X_{u_r}.$$
\end{definition}
\smallskip

\begin{remark}\label{dense}
Richardson varieties are the special case of intersection varieties
when $r=2$.  Let $B^{-}= w_0 B w_0$ be the opposite Borel subgroup of
$G$ defined by the property that $B \cap B^{-} = T$.  The $G$-orbit of
$(B, B^-)$ is a dense open orbit under the action of $G$ on $G/B
\times G/B^-$.  Consequently, for any choice of a pair $(g_{1},g_{2})$
in this orbit, the intersection of two translated Schubert varieties
defined with respect to the pair is isomorphic.  Hence, for $r=2$, the
generality condition simply means that $(g_{1},g_{2})$ should belong
to the dense open orbit.  In particular, $R(u_{1},u_{2}; g_{\bullet})$
is isomorphic to $R(u_{1},u_{2})$ for general
$g_{\bullet}=(g_{1},g_{2})$.
\end{remark}
\smallskip

\begin{remark}
When $r>2$, it is hard to characterize the $g_{\bullet}$ that are
sufficiently general.  Two general intersection varieties $R(u_1,
\dots, u_r; g_{\bullet})$ and $R(u_1, \dots, u_r; g_{\bullet}')$ are
not necessarily isomorphic or even birational to each other. For
example, let $u \in W/W_P$ be the element indexing the divisor class
for the Grassmannian $G(2,5)$. Then $R(u,u,u,u,u; g_{\bullet})$ is an elliptic curve.
 As $g_{\bullet}$ varies, all $j$-invariants occur in this family of
elliptic curves. Since two elliptic curves with different
$j$-invariants are not birational to each other, we get examples of
intersection varieties $R(u,u,u,u,u; g_{\bullet})$ and $R(u,u,u,u,u;
g_{\bullet}')$ in $G(2,5)$ that are not isomorphic (or even
birational) to each other, see \cite{coskun:DP}.
\end{remark}
\smallskip


It is well-known that Richardson varieties are reduced and irreducible
\cite{richardson}.  However, for $r>2$ the intersection varieties may
be reducible. For example, if $\sum l(u_i) = \dim(G/Q)$, then the
intersection variety $R(u_1, \dots, u_r, g_{\bullet})$ consists of
finitely many points, where the number of points is given by the
intersection number $\prod_{i=1}^r [X_{u_i}]$ of the Schubert classes. Nevertheless, every connected component of an intersection variety is normal and has rational singularities \cite[Lemma 4.1.2]{Brion:flag}.

Kleiman's Transversality Theorem \cite{kleiman} is the key tool for
characterizing the singular loci of intersection varieties.  We recall
the original statement of the theorem for the reader's convenience.
\smallskip

\begin{theorem}[Kleiman, \cite{kleiman}] Let $X$ be an integral
algebraic scheme with a transitive action of the integral algebraic
group $G$. Let $f: Y \rightarrow X$ and $g: Z \rightarrow X$ be two
maps of integral algebraic schemes. For any point $s \in G$, let $sY$
denote the $X$-scheme given by the map $y \mapsto s f(y)$.

\begin{enumerate}
\item [(i)] There exists a dense Zariski open subset $U$ of $G$ such
that for $s \in U$, the fibered product $sY \times_X Z$ is either
empty or equi-dimensional of dimension $\dim(Y) + \dim(Z) - \dim(X)$.
\item [(ii)]Assume the characteristic of the ground field is zero.
If $Y$ and $Z$ are smooth, then there exists a dense open subset $V$
of $G$ such that for $s \in V$, the fibered product $(sY) \times_X Z$
is smooth.
\end{enumerate}
\end{theorem}
\smallskip

\begin{remark}
The proof of part (ii) in Kleiman's Theorem uses generic
smoothness. This is the main reason why we are working over $\CC$.
Kleiman gives a specific example where (ii) fails if the ground field
has positive characteristic and $X$ is a Grassmannian variety
\cite[9.Example]{kleiman}.
\end{remark}
\smallskip

There exist many variations of Kleiman's Theorem in the original paper
and in the literature.  Below we spell out the variation we need for
Richardson varieties and intersection varieties in general.  This
variation is similar to the statement in \cite[Theorem 17.22]{harris}.

\begin{theorem}\label{thm:kleiman.variation}
Let $G$ be an algebraic group acting transitively on a smooth
projective variety $X$.  Let $Y$ and $Z$ be two subvarieties of $X$.
Then for any general translate $gY$ of $Y$, we have
\[
(gY\cap Z)^{sing} = ((gY)^{sing} \cap Z) \cup ((gY) \cap Z^{sing}).
\]
\end{theorem}

\begin{proof}
Since $Y$ and $Z$ are subvarieties of $X$, both map into $X$ by
inclusion.  Furthermore, for any $g\in G$ the fibered product
\[
(gY) \times_X Z = \{(gy,z)\in gY \times Z : gy=z \} = gY \cap Z.
\]
Applying part (ii) of Kleiman's Transversality Theorem to the smooth
loci 
\begin{align*}
(gY)^{sm} &= gY - (gY)^{sing}\\
Z^{sm} &= Z - Z^{sing}
\end{align*}
we conclude
that $(gY)^{sm} \cap Z^{sm}$ is smooth, provided $g$ is sufficiently
general.  Therefore, $(gY\cap Z)^{sing} \subset ((gY)^{sing} \cap Z)
\cup (gY \cap Z^{sing})$.

Conversely, by part (i) of Kleiman's Transversality Theorem, $Z$
intersects a general translate $gY$ of $Y$ properly, so we can assume
\[
dim((gY) \cap Z) = dim(gY) + dim(Z) - dim(X).
\]
We claim that a
proper intersection of two varieties cannot be smooth at a point where
one of the varieties is singular.  This claim is verified by the
following computation.  Let $p \in Z^{sing}$.  Then Lemma
\ref{tangent} below implies that
\begin{align*}
dim(T_{p}(gY \cap Z)) & = dim(T_{p}(gY) \cap T_{p}(Z))\\
& \geq dim(T_{p}(gY)) + dim(T_{p}(Z)) - dim(X) \hspace{.2in} \text{(since $X$ is smooth)}\\
& > dim(gY) + dim(Z) - dim(X) \hspace{.8in} \text{(since $p\in Z^{sing}$)}\\
& = dim(gY \cap Z).
\end{align*}
Therefore, $p\in (gY\cap Z)^{sing}$.  Similarly, $p \in (gY)^{sing}$
implies $p \in (gY\cap Z)^{sing}$.  This proves $((gY)^{sing} \cap Z)
\cup (gY \cap Z^{sing}) \subset (gY\cap Z)^{sing}$. Hence, we
conclude
$$(gY\cap Z)^{sing} = ((gY)^{sing} \cap Z) \cup ((gY) \cap Z^{sing}).$$
\end{proof}
\smallskip

\begin{lemma}\label{tangent}
Let $Y$ and $Z$ be two subvarieties of a smooth projective variety $X$
and let $p \in Y \cap Z$. Then $T_p (Y \cap Z) = T_p Y \cap T_p Z$ in
$T_pX$.
\end{lemma}
\smallskip

\begin{proof}
Since this is a local question, we may assume that $X$ is affine space
and $Y$ and $Z$ are affine varieties. Let $I(Y)$ and $I(Z)$ denote the
ideals of $Y$ and $Z$, respectively. Then $I(Y \cap Z) = I(Y) +
I(Z)$. Let $f_1, \dots, f_n$ be generators of $I(Y)$ and $g_1, \dots,
g_m$ be generators of $I(Z)$. Then $f_1, \dots, f_n, g_1, \dots, g_m$
generate $I(Y) + I(Z)$. The Zariski tangent spaces $T_p(Y)$, $T_p(Z)$
and $T_p(Y \cap Z)$ are the kernels of the matrices $$M= \left(
\frac{\partial f_i}{\partial t_j}(p) \right), N= \left( \frac{\partial
g_i}{\partial t_j}(p) \right), L= \left( \begin{array}{c} M \\ N
\end{array} \right),$$ respectively, where $t_j$ denote the
coordinates on affine space. It is now clear that the kernel of $L$ is
the intersection of the kernels of $M$ and $N$. 
\end{proof}
\smallskip

In the next proposition, we specialize
Theorem~\ref{thm:kleiman.variation}, to $X=G/\QP$.  Given a variety
$Y$ in $G/\QP$, let $[Y]$ denote the cohomology class of $Y$ in
$H^*(G/\QP, \ZZ)$.

\begin{proposition}\label{intersection}
The singular locus of the intersection variety $R(u_1, \dots, u_r;
g_{\bullet})$ is 
\[
\bigcup_{i=1}^r \bigg( g_{i}X_{u_i}^{sing} \cap R(u_1, \dots, u_r;
g_{\bullet}) \bigg)
=
\bigcup_{i=1}^r \left( g_{i}X_{u_i}^{sing} \cap \bigcap_{j=1}^r
g_{j}X_{u_j}\right).
\]
Hence, the singular locus of an intersection variety is a
union of intersection varieties. Furthermore, $R(u_1, \dots, u_r; g_{\bullet})$ is
smooth if and only if $$[X_{u_i}^{sing}] \cdot \prod_{j\not= i}
[X_{u_j}] = 0$$ for every $1 \leq i \leq r$.
\end{proposition}
\smallskip

\begin{proof} Applying
Theorem~\ref{thm:kleiman.variation} and induction on $r$ to 
$g_{1} X_{1} \cap \dotsb \cap g_{r}X_{r}$, it follows that $$(g_1 X_1 \cap
\cdots \cap g_r X_r)^{sing} = \bigcup_{i=1}^r \left( (g_iX_i)^{sing}
\cap \bigcap_{j=1}^r g_j X_j \right)$$ provided that the tuple $(g_1,
\dots, g_r)$ is general in the sense of Kleiman's Transversality
Theorem. In particular, taking $X_i$ to be the Schubert variety 
$X_{u_i} \in X = G/\QP$ we recover the first statement in Proposition
\ref{intersection}.

Since the singular locus of a Schubert variety is a union of Schubert
varieties corresponding to certain smaller dimensional $B$-orbits, we
conclude that the singular locus of a Richardson variety is a union of
Richardson varieties. In particular, by part (i) of Kleiman's
Transversality Theorem, we may assume that all the intersections
\[
g_{i}X_{u_i}^{sing} \cap \bigcap_{j \not= i}
g_{j}X_{u_j}
\]
are dimensionally proper. The cohomology class of this intersection is
the cup product of the cohomology classes of each of the Schubert
varieties. Hence, this intersection is empty if and only if its
cohomology class is zero.  We conclude that the singular locus of an
intersection variety is empty if and only if the cohomology classes
$[X_{u_i}^{sing}] \cdot \prod_{j \not= i} [X_{u_j}] = 0$ for all $1
\leq i \leq r$. This concludes the proof of Proposition
\ref{intersection}.
\end{proof}
\smallskip

Specializing to the case $r=2$, we obtain the following characterization of the singular loci of Richardson varieties.
\smallskip

\begin{corollary}\label{richardson}
Let $R(u,v) = X_u \cap X^v$ be a non-empty Richardson variety in
$G/{\QP}$. Let $X_{u}^{sing}$ and $X^{v}_{sing}$ denote the singular loci
of the two Schubert varieties $X_u$ and $X^v$, respectively. Then the
singular locus of $R(u,v)$ is a union of Richardson varieties
$$R(u,v)^{sing}=(X_{u}^{sing} \cap X^v) \cup (X_u \cap X^{v}_{sing}).$$ In
particular, $R(u,v)$ is smooth if and only if the cohomology classes
$[X_{u}^{sing}]\cdot [X^v] = 0 $ and $[X_u] \cdot [X^{v}_{sing}] = 0$
in the cohomology ring $H^*(G/{\QP}, \ZZ)$.
\end{corollary}

\smallskip

\begin{remark}
Checking the vanishing conditions in Corollary \ref{richardson} is
very easy once the singular locus is determined since these conditions
require only the product of pairs of Schubert classes to vanish.  This
is equivalent to testing the relations between pairs of elements in
Bruhat order.  However, in general for $r>2$, checking the vanishing
conditions in Proposition \ref{intersection} is a hard problem which
requires computing Schubert structure constants.  There are many
techniques for doing these computations; see for example
\cite{BGG,B4,Brion:flag,coskun:LR,Dem,duan,duan-zhao,Fulton-book,KK}
and references within those.  It is an interesting open problem to
efficiently characterize all triples $u,v,w \in W/W_{\QP}$ such that
$[X_{u}][X_{v}][X_{w}]=0$.  Purbhoo has given some necessary
conditions for vanishing in \cite{purbhoo}.
\end{remark}
\smallskip


Recall, that the typical way of studying the singularities of Schubert
varieties in the literature relies on tests for the $T$-fixed points
of the Schubert variety. Even for Richardson varieties $R(u,v)$ in the
Grassmannian, there may not be any torus fixed points in the smooth
locus of $R(u,v)$.  Below we give an example.  

\begin{example} 
Consider the Grassmannian variety $G(3,8)$ of $3$-planes in
$\mathbb{C}^{8}$.  Fix a basis $e_1, \dots, e_8$ of $V$. The $T$-fixed
points in $G(3,8)$ are the subspaces spanned by three distinct basis
elements $\{e_{w_{1}}, e_{w_{2}}, e_{w_{3}}\}$. These points are
indexed by permutations $w=[w_{1},\dotsc , w_{8}] \in S_{8}$ such that
$w_{1}<w_{2}<w_{3}$ and $w_{4}<\dotsb < w_{8}$.  Such permutations
could be denoted simply by $(w_{1},w_{2},w_{3})$.  Furthermore, these
permutations are in bijection with partitions that
fit in a $3 \times (8-3)$ rectangle.  So, $u=(4,6,8)=[4,6,8,1, 2, 3, 5, 7]$ is a
$T$-fixed point of $G(3,8)$ and it corresponds with the partition
$(2,1,0)$.  See \cite{Fulton-book} for more detail.

Let both $u$ and $v$ be the permutation $u=(4,6,8)$, or equivalently the partition
$(2,1,0)$ so that $X_{u}$ is isomorphic to
$X^{v}$ but in opposite position.  Consider $R(u,v) \subset G(3,8)$
defined in terms of flags $F_{\bullet}=(F_{1},\dotsc , F_{8})$ and
$G_{\bullet}=(G_{1},\dotsc , G_{8})$ where $F_i$ is the span of the
first $i$ basis elements and $G_i$ is the span of the last $i$ basis
elements.  The Schubert variety $X_u$ is the set of $3$-dimensional
subspaces that intersect $F_4, F_6$ and $F_8$ in subspaces of
dimensions at least $1,2$ and $3$, respectively. Similarly, $X^v$ is
the set of $3$-dimensional subspaces that intersect $G_4, G_6$ and
$G_8$ in subspaces of dimensions at least $1,2$ and $3$, respectively.

The singular locus of $X_{u} \subset G(k,n)$ is the union of Schubert
varieties indexed by all the partitions obtained from the partition
corresponding with $u$ by adding a maximal hook in a way that the
remaining shape is still a partition \cite[Thm 9.3.1]{BLak}.  Thus,
$X_{(4,6,8)}^{sing} = X_{(3,4,8)} \cup X_{(4,5,6)}$.  In terms of
flags, $\Lambda \in X_{u}$ is a singular point if $\dim(\Lambda \cap
F_{4}) \geq 2$ or $\dim(\Lambda \cap F_{6}) =3$.  Similarly, $\Lambda$
is singular for $X^{v}$ if $\dim(\Lambda \cap G_{4}) \geq 2$ or
$\dim(\Lambda \cap G_{6}) =3$.

The $T$-fixed points of $R(u,v)$ consist of subspaces that are spanned
by $e_{i_1} \in F_4, e_{i_2} \in F_6 \cap G_6$ and $e_{i_3} \in G_4$.
We claim that each of these $T$-fixed points is singular in either
$X_{u}$ or $X^{v}$, so by Corollary~\ref{richardson} they are all
singular in $R(u,v)$.  The claim holds since any basis vector in $F_6
\cap G_6$ is also contained in either $F_4$ or $G_4$. We conclude that
the smooth locus of $R(u,v)$ does not contain any torus fixed
points. 
\end{example}
\smallskip

In Remark~\ref{rem:t.fixed}, we will characterize the Richardson
varieties in the Grassmannian that contain a torus fixed smooth point.

More generally, when $G/{\QP}$ is the Grassmannian $G(k,n)$, Corollary
\ref{richardson} implies a nice, geometric characterization of the
smooth Richardson varieties. This characterization is essentially
proved but not explicitly stated in \cite{kreiman}.  \smallskip

\begin{definition}\label{definition-segre}
The {\em Segre product} of $r$ Grassmannians $G(k_1, n_1)\times \cdots \times G(k_r, n_r)$  in $G(\sum k_i, n)$, with $n \geq \sum n_i$, is the image under the direct sum map $$(W_1, \dots, W_r) \mapsto (W_1 \oplus \cdots \oplus W_r).$$
\end{definition}
\smallskip

A Segre product of Grassmannians is a Richardson variety in $G(k,n)$.
Specifically, the Segre product $G(k_1, n_1)\times \cdots \times
G(k_r, n_r)$ is the intersection of the two opposite Schubert
varieties $X_{u} \cap X^{v}$ where
$$
u= (n_1-k_1+1, n_1 - k_1 + 2, \dots, n_1, \dots, \sum_{i=1}^r n_i -
k_r + 1, \dots, \sum_{i=1}^r n_i)
$$ 
and 
$$v = (n - \sum_{i=1}^{r-1} n_i-k_r+1, n - \sum_{i=1}^{r-1} n_i-k_r+2,
\dots, n - \sum_{i=1}^{r-1} n_i, \dots, n-k_1+1, \dots, n-1, n).
$$
This Richardson variety can also be realized as the intersection of
two Schubert varieties with respect to partial flags. Let $F_{j} = V_1
\oplus \cdots \oplus V_j$, where $V_i$ is the $n_i$ dimensional vector
space defining $G(k_i, n_i)$.  Let $E$ be a vector subspace of $V$
complementary to $F_r$.  Let $G_j = E \oplus V_r \oplus
\cdots \oplus V_{r-j+1}$.  Then we have
$$X_u = \{ W \in G(k,n) \ | \
\dim(W \cap F_j) \geq \sum_{i=1}^j k_i \} 
$$ 
and 
$$X^v = \{ W \in G(k,n) \ | \ \dim(W \cap G_j) \geq \sum_{i=r-j+1}^r
k_i \}, $$

\smallskip

\begin{corollary}\label{segre}
A Richardson variety in $G(k,n)$ is smooth if and only if it is a Segre product of Grassmannians. 
\end{corollary}
\smallskip

\begin{remark}
By taking $X^v$ to equal $G(k,n)$, we see that Corollary \ref{segre} generalizes the well-known fact that a Schubert variety $X_u$ in $G(k,n)$ is smooth if and only if $X_u$ is a sub-Grassmannian. 
\end{remark}
\smallskip

\begin{proof}
A Segre product of Grassmannians is clearly smooth since it is the
Segre embedding of the product of Grassmannians $G(k_1, n_1) \times
\cdots \times G(k_r, n_r)$ in $G(k,n)$ with $k = \sum k_i$ and $n \geq
\sum n_i$.  Conversely, suppose that $R(u,v)$ is a smooth Richardson
variety in $G(k,n)$. Observe that $X_{u}^{sing} \cap
X^{v} = \emptyset$ if and only if $u_i + v_{n-i} \leq n$ for
every $1\leq i \leq k$ for which $u_{i+1} \not= u_i +1$. Similarly,
$X_{u} \cap X^{v}_{sing} = \emptyset$ if and only
if $v_i + u_{n-i} \leq n$ for every $1\leq i \leq k$ for which $v_{i+1} \not= v_i
+1$. Suppose first that $u_{i+1} = u_i+1$ for all $1 \leq i \leq
k-1$. Then the Richardson variety is isomorphic to a Schubert variety
and is smooth if and only if it is a Grassmannian. The corollary now
is immediate by induction on the number of times $u_{i+1} \not=u_i
+1$. Suppose $u_{j+1} \not= u_j +1$. Then $u_j + v_{n-j} \leq n$ and
the Richardson variety is a product of two Richardson varieties in
$G(j, F_{u_j}) \times G(n-j,G_{v_{n-j}})$ and is smooth if and
only if each factor is smooth.
\end{proof}
\smallskip

\begin{remark}
Finally, we note that the proof given in \cite[Remark 7.6.6]{kreiman} for determining the multiplicities of Richardson varieties in minuscule partial flag varieties generalizes by induction to intersection varieties. Let $ R(u_1, \dots, u_r; g_{\bullet})$ be an intersection variety in a minuscule partial flag variety. Then
$$\mbox{mult}_p \big(R(u_1, \dots, u_r; g_{\bullet})\big) =
\prod_{i=1}^r \mbox{mult}_p (g_{i}X_{u_i}).$$
\end{remark}




\section{Projection Varieties}\label{s:main-1}

In this section, we prove that projection varieties have rational
singularities.  This claim follows from a general fact, which we prove
below, about the images of certain subvarieties under Mori
contractions.  We refer the reader to \cite{kollar:mori} for more
detail about Mori theory and rational singularities.

\begin{definition}
A variety $X$ has {\em rational singularities} if there exists a
resolution of singularities $f: Y \rightarrow X$ such that $f_* \OO_Y
= \OO_X$ and $R^if_* \OO_Y = 0$ for $i>0$.
\end{definition}

\begin{remark}
A variety with rational singularities is normal. Moreover, for every resolution $g:Z \rightarrow X$, we have $g_* \OO_Z = \OO_X$ and $R^i g_* \OO_Z = 0$ for $i>0$ (see \cite{kollar:mori}).
\end{remark}

The map $\pi_Q: G/P \rightarrow G/Q$ is a Mori contraction.  To deduce Theorem \ref{main-1} we will apply the following general theorem.  

\smallskip

\begin{theorem}\label{Mori}
Let $X$ be a smooth, projective variety. Let $\pi: X \rightarrow Y$ be a Mori contraction defined by the line bundle $M = \pi^* L$, where $L$ is an ample line bundle on $Y$. Let $Z \subset X$ be a normal, projective subvariety of $X$  with rational singularities. Assume that 
\begin{enumerate}
\item $H^i(Z, M^{\otimes n}|_Z)= 0$ for all $i>0$ and all $n \geq 0$.
\item The natural restriction map $H^0(X, M^{\otimes n}) \rightarrow H^0(Z, M^{\otimes n}|_Z)$ is surjective for all $n \geq 0$.
\end{enumerate} 
Then $W = \pi(Z)$, with its reduced induced structure, is normal and has rational singularities. 
\end{theorem} 
\smallskip

\begin{proof}
Denote the restriction of the map $\pi$ to $Z$ also by $\pi$. To simplify notation, we will denote the restriction of the line bundles $M$ to $Z$ and $L$ to $W$ again by $M$ and $L$. 
\smallskip

\noindent Step 1.  We first show that $R^i \pi_* \OO_Z = 0$ for $i>0$. Since $L$ is ample on $W$, by Serre's Theorem, $R^i \pi_* \OO_Z \otimes L^{\otimes n}$ is generated by global sections for $n >>0$ (\cite{hartshorne}, II.5.17). Therefore, to show that $R^i \pi_* \OO_Z = 0$, it suffices to show that $H^0(W, R^i \pi_* \OO_Z \otimes L^{\otimes n}) = 0$ for $n>>0$. Similarly, by Serre's Theorem, $H^j(W, R^i \pi_* \OO_Z \otimes L^{\otimes n})=0$ for all $j>0$ and all $n>>0$ (\cite{hartshorne}, III.5.2). Since only finitely many sheaves $R^i \pi_* \OO_Z$ are non-zero, we may choose a number $N$ such that for all $n\geq N$, $R^i \pi_* \OO_Z \otimes L^{\otimes n}$ is globally generated and has no higher cohomology for all $i$. 
\smallskip

Given a coherent sheaf $\mathcal{F}$ on $Z$, the Leray spectral
sequence expresses the cohomology of $\mathcal{F}$ in terms of the
cohomology of the higher direct image sheaves $R^i \pi_* \mathcal{F}$
on $W$. More precisely, the spectral sequence has $E_2^{p,q} =
H^p(W, R^q \pi_* \mathcal{F})$ and abuts to $H^{p+q}(Z,
\mathcal{F})$. We apply the  spectral sequence to $\mathcal{F} =
M^{\otimes n}$ for $n \geq N$. Since $M = \pi^* L$, by the projection
formula, $R^i \pi_* M^{\otimes n} = R^i \pi_* \OO_Z \otimes L^{\otimes
n}$. Since $n$ is chosen so that $H^p(W, R^q \pi_* \OO_Z \otimes
L^{\otimes n})=0$ for $p>0$, we conclude that the  spectral
sequence degenerates at the $E_2$ term. Consequently, $H^0(W, R^i
\pi_* \OO_Z \otimes L^{\otimes n}) \cong H^i(Z, M^{\otimes n})$. Since
by assumption, $H^i(Z, M^{\otimes n})=0$ for $i>0$ and $n>0$, we
conclude that $H^0(W, R^i \pi_* \OO_Z \otimes L^{\otimes n})=0$ for
$i>0$ and $n \geq N$. Therefore, $R^i \pi_* \OO_Z =0$ for $i>0$.
\smallskip

\noindent Step 2. We next show that $\pi_* \OO_Z = \OO_W$. In particular, this implies that the Stein factorization of the map $\pi: Z \rightarrow W$ is trivial (\cite{hartshorne}, III.11.5). Therefore, the fibers of the map $\pi$ are connected and $W$ is normal.  Let $\mathcal{F}$ be the cokernel of the natural injection from $\OO_W$ to $\pi_* \OO_Z $. We thus obtain the exact sequence $$(*) \ \ 0 \rightarrow \OO_W \rightarrow \pi_* \OO_Z \rightarrow \mathcal{F} \rightarrow 0.$$  We want to show that $\mathcal{F}=0$. Since $L$ is ample, $\mathcal{F} \otimes  L^{\otimes n}$ is globally generated for $n >>0$. Hence, it suffices to show $H^0 (W, \mathcal{F} \otimes L^{\otimes n})=0$ for $n >>0$. Using the long exact sequence of cohomology associated to the exact sequence $(*)$ and the fact that $H^1 (W, L^{\otimes n})=0$ for $n >>0$, to conclude that $\mathcal{F}=0$, it suffices to show that $h^0(W, L^{\otimes n}) = h^0(W, \pi_* \OO_Z \otimes L^{\otimes n})$ for $n > > 0$. 
\smallskip

Consider the exact sequence $$0 \rightarrow I_W \rightarrow \OO_Y \rightarrow \OO_W \rightarrow 0.$$ Tensoring the exact sequence by $L^{\otimes n}$ for $n>>0$, we get that the restriction map $H^0(Y, L^{\otimes n}) \rightarrow H^0(W, L^{\otimes n})$ is surjective. Since the map $\pi: X \rightarrow Y$ is a Mori contraction, $\pi_* \OO_X = \OO_Y$. Consequently, by the projection formula $$H^0(Y, L^{\otimes n}) = H^0(Y, \pi_* \OO_X \otimes L^{\otimes n}) = H^0(X, M^{\otimes n}).$$ The inclusion of $Z$ in $X$ and $W$ in $Y$ gives rise to the commutative diagram

\begin{align*}
 &H^0(Y, L^{\otimes n})  =  H^0(Y, \pi_* \OO_X \otimes L^n) = H^0(X, M^{\otimes n})&  \\
&\ \ \ \ \ \ \downarrow \ \ \ \ \ \ \ \ \ \ \ \ \ \ \ \ \ \ \ \ \ \ \ \ \ \ \ \ \ \ \ \ \ \ \ \ \swarrow& \\
&H^0(W, L^{\otimes n}) \ \ \ \ \ \ \ \ \ \ \ \ \ \ \  \ \ \ \ \ H^0(Z, M^{\otimes n})&\\
&\ \ \ \ \ \ \ \downarrow  \ \ \ \ \ \ \ \  \ \ \ \ \ \ \ \ \ \ \ \ \ \ \swarrow& \\
&H^0(W, \pi_* \OO_Z \otimes L^{\otimes n})  & 
\end{align*}

By assumption, the restriction map $H^0(X, M^{\otimes n}) \rightarrow H^0(Z, M^{\otimes n})$ is surjective for all $n \geq 0$. $H^0(Z, M^{\otimes n}) \cong H^0(W, \pi_* \OO_Z \otimes L^{\otimes n})$. In particular, it follows that $H^0(Y, L^{\otimes n}) \rightarrow H^0(W, \pi_* \OO_Z \otimes L^{\otimes n})$ is surjective. Therefore, the map $H^0(W, L^{\otimes n}) \rightarrow H^0(W, \pi_* \OO_Z \otimes L^{\otimes n})$ must be surjective. Since by the exact sequence $(*)$, it is also injective, we conclude that $h^0(W, L^{\otimes n}) = h^0(W, \pi_* \OO_Z \otimes L^{\otimes n})$ for $n > > 0$. This concludes the proof that $\OO_W \cong \pi_* \OO_Z$ and that $W$ is normal.
\smallskip

\noindent Step 3. Finally, to conclude that $W$ has rational singularities, we simply apply a theorem of Koll\'{a}r. First, observe that since $Z$ has rational singularities by assumption, for any desingularization $\rho: U \rightarrow Z$, we have that  $\rho_* \OO_U \cong \OO_Z$ and $R^i \rho_* \OO_U = 0$ for $i>0$. Hence, considering $\phi = \pi \circ \rho: U \rightarrow W$, we have that $\phi_* \OO_U = \pi_* (\rho_* \OO_U) =  \OO_W$ and $R^i \phi_* \OO_U =0$ for $i>0$. In particular, the Stein factorization of $\phi$ is trivial and the geometric generic fiber of $\phi$ is connected. Koll\'{a}r's Theorem 7.1 in \cite{kollar} then guarantees that $W$ has rational singularities. This concludes the proof.
\end{proof}

\smallskip

\begin{remark}
The proof of Theorem \ref{Mori} does not use the full strength of the hypotheses. In assumptions (1) and (2), it is not necessary to require the vanishing of higher cohomology and the surjectivity of the restriction map for all $n\geq 0$. It suffices to assume these only for sufficiently large $n$.  
\end{remark}
\smallskip

Theorem \ref{Mori} reduces understanding singularities of certain subvarieties of flag varieties to the vanishing of cohomology. 
The higher cohomology groups of the restriction of NEF line bundles on flag varieties to Schubert or Richardson varieties vanish. Hence, Theorem \ref{Mori} is a very useful tool in the context of Schubert geometry. For instance, Theorem \ref{Mori} immediately implies Theorem \ref{main-1}.
\smallskip

\begin{proof}[Proof of Theorem \ref{main-1}]
Let $Q \subset G$ be a parabolic subgroup containing $P$. Let $P_i$, $i=1, \dots, j$, be the maximal parabolic subgroups containing $Q$. Let $\pi_Q: G/P \rightarrow G/Q$ and $\pi_{P_i}: G/Q \rightarrow G/P_i$ denote the natural projections.  
In Theorem \ref{Mori} take $$X= G/P \ \ \mbox{and} \ \ Y=G/Q.$$ Let $L$ be the ample  line bundle on $Y$ defined by $L = L_{i_1} \otimes L_{i_2} \otimes \cdots \otimes L_{i_j}$, where $L_{i_s}$ is the pull-back of the ample generator of the Picard group of $G/P_i$ under the natural projection map $\pi_{P_i}$. Let $M = \pi_{Q}^* L$. Then $M$ is NEF and defines the projection $\pi_{Q}: X \rightarrow Y.$
Let $Z$ be the Richardson variety $R(u,v)$. It is well-known that Richardson varieties are normal with rational singularities (\cite{Brion:flag} Theorem 4.1.1). Furthermore, the higher cohomology of the restriction of a NEF line bundle on the flag variety to a Richardson variety vanishes and the restriction map on global sections is surjective \cite{Brion:flag} Theorem 4.2.1 (ii) and Remark 4.2.2. Hence, all the assumptions of Theorem \ref{Mori} are satisfied. We conclude that the projection variety $\pi_Q (Z)$ is normal and has rational singularities. 
\end{proof}


\section{Singularities of Grassmannian projection varieties}\label{s:grassmannian.projetion.varieties}

In this section, we discuss the singularities of projection varieties in Type A Grassmannians. We first characterize projection varieties in $G(k,n)$ without reference to a projection from a flag variety. Given a projection variety, we then exhibit a minimal Richardson variety projecting to it. This Richardson variety is birational to the projection variety and the projection map does not contract any divisors. This description allows us to characterize the singular loci of projection varieties.  We first begin by introducing some notation. 
\smallskip

\begin{notation}
Let $0<k_1 < k_2 < \cdots <k_m <n$ be an increasing sequence of positive integers less than $n$. We set $k_0=0$ and $k_{m+1} = n$. Let $Fl(k_1, \dots, k_m; n)$ denote the partial flag variety parameterizing partial flags  $(V_1 \subset \cdots \subset V_m)$ of length $m$ in $V$ such that $V_i$ has dimension $k_i$.  

The cohomology of $Fl(k_1, \dots, k_m; n)$ admits a $\ZZ$-basis
generated by the classes of Schubert varieties. Schubert varieties in
$F(k_1, \dots, k_m; n)$ are parameterized by partial permutations in
$\mathfrak{S}_n$ with $k_m$ entries and at most $m$ descents at
positions $k_1, k_2, \dots, k_m$. We record these permutations as a
list of $k_m$ distinct positive integers $u= (u_1, u_2, \cdots,
u_{k_m})$ less than or equal to $n$ such that $u_j < u_{j+1}$ unless
$j = k_i$ for some $1\leq i \leq m$. Given an entry $u_i$ in the
permutation, there exists a unique $l$ such that $k_{l-1} < i \leq
k_l$. We say that the {\em color} $c_i$ of the entry $u_i$ is $l$.
Geometrically, the entries of the permutation record the dimensions of
the elements in flag $F_{\bullet}=(F_{1},\ldots, F_{n})$ defining the
Schubert variety $X_u$ where a jump in dimension occurs and the
corresponding color records the minimal $j$ for which the dimension of
$V_j$ is required to increase:
$$X_u (F_{\bullet}) = \{ (V_1, \dots, V_m) \in Fl(k_1, \dots, k_m; n)
\ | \ \dim(V_j \cap F_{u_i}) \geq \# \{u_l \leq u_i \ | \ c_l \leq j
\}\}.$$ It is convenient to assign a multi-index $I_u(i) = (s_1^i,
\dots, s_m^i)$ to each entry in a permutation by letting $$s_j^i = \#
\{u_l \leq u_i \ | \ c_l \leq j \}.$$ In particular, using the
multi-indices, the definition of a Schubert variety can be expressed
more compactly:
$$X_u (F_{\bullet}) = \{ (V_1, \dots, V_m) \in Fl(k_1, \dots, k_m; n) \ | \ \dim(V_j \cap F_{u_i}) \geq s_j^i \}.$$
\end{notation}
\smallskip

\begin{example}
Let $u= (1,4,8, 3, 9, 2,7)$ be a permutation for $F(3,5,7; 9)$. Then
the color of the entries $1,4,8$ is $1$, the color of the entries $3,9$ is
$2$ and the color of the entries $2,7$ is $3$. The multi-indices are
$I_u(1) = (1,1,1), I_u(2)= (2,3,4), I_u(3)= (3,4,6), I_u(4)= (1,2,3), I_u(5)=
(3,5,7), I_u(6)= (1,1,2), I_u(7)= (2,3,5)$. The corresponding Schubert
variety parameterizes flags $(V_1, V_2, V_3)$ such that $V_3$ is
required to intersect $F_1, F_2, F_3, F_4$, $ F_7, F_8, F_9$ in subspaces
of dimension at least $1,2,3,4,5,6,7$, respectively. $V_2$ is required
to intersect $F_1, F_3, F_4, F_8, F_9$ in subspaces of dimension at
least $1,2,3,4,5$, respectively. Finally, $V_1$ is required to
intersect $F_1, F_4, F_8$ in subspaces of dimension at least $1,2,3$,
respectively.
\end{example}
\smallskip

Recall that a {\em Grassmannian projection variety} is the projection of a Richardson variety $R(u,v)$ in $Fl(k_1, \dots, k_m; n)$ to $G(k_m, n)$ under the natural projection map.

\begin{remark}
There may be many different ways of realizing the same projection
variety as projections of different Richardson varieties. For example,
take two successive projections of flag varieties $$\mathcal{F}_1=
Fl(k_1, k_2, k_3; n) \stackrel{\pi_1}{\longrightarrow} \mathcal{F}_2=
Fl(k_2, k_3; n) \stackrel{\pi_2}{\longrightarrow} G(k_3,n).$$ Given a
Richardson variety $R(u,v) \subset \mathcal{F}_2$,
$\pi_1^{-1}(R(u,v))=R(u',v')$ is a Richardson variety in
$\mathcal{F}_1$. Hence, the projection variety $\pi_2(R(u,v))$ may
also be realized as the projection variety $\pi_2 \circ \pi_1
(R(u',v'))$.  In particular, the Grassmannian projection varieties all
come from projections of Richardson varieties in the complete flag
variety.  This proves that rank varieties are examples of the (closed) positroid varieties defined in
\cite[Section 5]{KLS}.  Note, that
lifting a Richardson variety in $\mathcal{F}_{1}$ up to
$\mathcal{F}_{2}$ and then projecting it to $G(k,n)$ need not be a
birational map.  
\end{remark}

\smallskip

It is convenient to have a characterization the projection varieties
without referring to the projection of a Richardson variety.  Such a
characterization can be obtained in terms of rank varieties.  We
recall the following notation from Section~\ref{s:intro}.  Let $V$ be
an $n$-dimensional vector space. Fix an ordered basis $e_1, \dots,
e_n$ of $V$.  Let $W= [e_i, e_j]$ be the vector space spanned by a
consecutive set of basis elements $e_i, e_{i+1}, \dots, e_j$.  Let
$l(W)= i$ and $r(W)= j$ be the smallest and largest index of the basis
elements contained in $W,$ respectively.  A {\em rank set} $M$ for
$G(k,n)$ is a set of $k$ vector spaces $M=\{ W_1, \dots, W_k \}$,
where each vector space is the span of (non-empty) consecutive
sequences of basis elements and $l(W_i) \not= l(W_j)$ and $r(W_i)
\not= r(W_j)$ for $i \not= j$. Given a rank set $M$, the corresponding
{\em rank variety} $X(M)$ is the subvariety of $G(k,n)$ defined by the
Zariski closure of the set of $k$-planes in $V$ that have a basis
$b_1, \dots, b_k$ such that $b_i \in W_i$ for $W_i \in M$.

\begin{remark}
Alternatively, one can define a rank variety $X(M)$ as the variety of
$k$-planes that intersect any vector space $W$ spanned by the ordered
basis in a subspace of dimension at least the number of $W_i \in M$
contained in $W$ \begin{align*}X(M) = \{\Lambda \in G(k,n) \ | \
\dim(\Lambda \cap W) \geq \#\{W_i \in M \ | \ W_i \subset W\} \\
\mbox{for every} \ W = <e_{i_1}, \dots, e_{i_s}> \}.\end{align*} These rank equations 
give rise to the terminology. In the sequel we will not need this fact, so we leave showing the equivalence of the two definitions to the reader.
\end{remark}

We can characterize projection varieties in $G(k,n)$ as rank varieties. 

\begin{theorem}\label{projection=rank}
$X \subset G(k,n)$ is a projection variety if and only if $X$ is a rank variety.
\end{theorem}

We will prove Theorem \ref{projection=rank} in several steps. We begin
by giving two algorithms. The first algorithm associates a Richardson
variety $R(u,v)(M)$ to every rank variety $X(M)$ in $G(k,n)$ in a
minimal way such that the projection of $R(u,v)(M)$ is $X(M)$. Given a
Grassmannian projection variety, the second algorithm associates to it
a rank set with the corresponding rank variety being equal to the
original projection variety.  

Throughout this proof, we fix two opposite flags
$F_{\bullet}=(F_{1},\dotsc , F_{n})$ and $G_{\bullet} =(G_{1},\dotsc
,G_{n})$ where $F_i$ is the span of the first $i$ basis elements $e_1,
\dots, e_i$ and $G_i$ in $G_{\bullet}$ is the span of the last $i$
basis elements $e_n, \dots, e_{n-i+1}$.  Then
$X_{u}= X_{u}(F_{\bullet})$ and $X^{v}=X_{v}(G_{\bullet})$.

\begin{algorithm}\label{rich}[{\bf Associating a Richardson variety to
a rank variety.}] In this algorithm, given a rank set $M=\{W_1, \dots,
W_k\}$ and its rank variety $X(M)$ in $G(k_m, n)$, we will associate a
Richardson variety $R(u,v)(M)$ in an appropriate flag variety $F(k_1,
\dots, k_m; n)$ such that the projection of $R(u,v)$ is $X(M)$. 
 \smallskip

\noindent {\bf Step 1: Associate a color to each vector space $W_i$.}
Let $m$ be the longest chain of subspaces $$W_{j_1} \supsetneq W_{j_2}
\supsetneq \cdots \supsetneq W_{j_m},$$ where each $W_{j_s} \in M$.
For a vector space $W_i \in M$, let $m_i$ be the length of the longest
chain
$$W_{j_1} \supsetneq W_{j_2} \supsetneq \cdots \supsetneq W_{j_{m_i}}
= W_i,$$ where $W_{j_s} \in M$. Assign $W_i$ the color $c_i = m - m_i
+ 1$. From now on we decorate the vector spaces in the rank set with
their color $W_i^{c_i}$.  \smallskip
 
 \noindent {\bf Step 2: Define two opposite Schubert varieties.}  Let
$k_j$ be the number of vector spaces in the rank set that are assigned
a color less than or equal to $j$. We define two Schubert varieties in
$F(k_1, \dots, k_m; n)$. Recall that $r(W_i)$ is the index of the
basis element with the largest index in $W_i$. Let $u$ be the
permutation defined by the numbers $r(W_i)$ listed so that those
corresponding to vector spaces of color $c$ all occur before those of
color $c+1$ and among those of the same color the numbers are
increasing. Similarly, recall that $l(W_i)$ is the index of the basis
element with the smallest index in $W_i$. Let $v$ be the permutation
defined by the numbers $n-l(W_i)+1$ listed so that those corresponding
to vector spaces of color $c$ all occur before those of color $c+1$
and among those of the same color the numbers are increasing. Let the
{\em minimal Richardson variety} $R(u,v)(M)$ associated to the rank
set $M$ be the Richardson variety $X_{u}\cap X^{v}= X_u (F_{\bullet}) \cap X_{v}(G_{\bullet})$ in $F(k_1, \dots, k_m; n)$.
\end{algorithm} 
\medskip

\begin{remark}
The requirement that $l(W_i) \not= l(W_j)$ and $r(W_i) \not= r(W_j)$
for $i \not= j$ guarantees that the numbers $r(W_i)$ and $n-l(W_i)+1$
are all distinct. Moreover, both permutations have at most $m$
descents at places $k_1, \dots, k_m$ by construction. Therefore,
Algorithm \ref{rich} produces well-defined partial permutations $u$
and $v$ for $F(k_1, \dots, k_m; n)$.
\end{remark}

\medskip

\begin{algorithm}\label{rank}[{\bf Associating a rank set to a
Richardson variety.}] Let $R(u,v) = X_u \cap X^v = X_{u}(F_{\bullet})
\cap X_{v}(G_{\bullet})$ be a Richardson variety in $F(k_1, \dots,
k_m; n)$. In this algorithm, we associate a rank set $M(R(u,v))$ to
$R(u,v)$ such that the projection of $R(u,v)$ to $G(k_m, n)$ is the
rank variety $X(M(R(u,v)))$.  \smallskip

\noindent {\bf Step 0.} Recall that each $u_i$ in the permutation $u$
is assigned a color $c_i$, where $c_i=l$ if $k_{l-1} < i \leq k_l$,
and a multi-index $I_u(i) = (s_1^i, \dots, s_m^i)$, where $s_j^i = \#
\{ u_l \leq u_i \ | \ c_l \leq j \}$. Similarly, each $v_i$ in the
permutation $v$ is assigned a color $d_i$ and a multi-index $I_v(i) =
(t_1^i, \dots, t_m^i)$. Given an entry $u_j$ in a permutation, recall
$u^{-1}(u_j)= j$ denotes the index of the entry. The color, index and
multi-index are assigned to $u_i$ or $v_i$ in a permutation for once
and for all and do NOT vary during the algorithm.  \smallskip

\noindent {\bf Step 1.} Let $U_{k_m}= \{u_i: 1\leq i\leq k_{m}\}$ be
the initial set of entries in the permutation $u$ and let $V_{k_m} =
\{v_i : 1\leq i\leq k_{m}\}$ be the analogous the set of entries for
$v$.  Let $M_0$ be the empty set. At each stage, we will remove an
element from each of $U_{k_m}$ and $V_{k_m}$ and add a vector space to
$M_0$ until we exhaust $U_{k_m}$ and $V_{k_m}$.  \smallskip

\noindent {\bf Initial Step.} Let 
\[
\alpha = \min_{u_i \in U_{k_m}} (u_i). 
\]
Let $c=c_{u^{-1}(\alpha)}$ be the color of the entry
$\alpha$.  For each $d \geq c$, let $$\beta^d = \max_{v_i \in V_{k_m}
\text{ s.t. } d_i \leq d} (v_i).$$ Let 
$\beta$ be the minimum  $\beta^{d}$ over all $d\geq c$.
%
%
Let $W_1 = F_{\alpha} \cap G_{\beta}$, where $F_{\bullet}$
and $G_{\bullet}$ are the two flags defining the Schubert
varieties $X_u$ and $X^v$, respectively as always.  Set $U_{k_m -1} =
U_{k_m} - \{\alpha\}$,\ 
$V_{k_m -1} = V_{k_m} - \{ \beta\}$ and $M_1 = \{ W_1 \}$.  \medskip

\noindent {\bf The Inductive Step.} Suppose we have defined
$M_\mathfrak{t}$ and are left with two subsets $U_{k_m -\mathfrak{t}}$
and $V_{k_m -\mathfrak{t}}$ of the entries from the permutations $u$
and $v$. Let $$\alpha= \min_{u_i \in
U_{k_m-\mathfrak{t}}} (u_i).$$ 
Let $c=c_{u^{-1}(\alpha)}$ as above. 
For each $d \geq c$, let
\[
\beta^d = \max_{v_i \in V_{k_m-\mathfrak{t}} \text{ s.t. } d_i \leq d} \ (v_i).
\]
Let 
$\beta$ be the minimum  $\beta^{d}$ over all  $d\geq c$ for which 
$t_d^{v^{-1}(\beta^d)}
\geq k_d - s_d^{u^{-1}(\alpha)} +1$.
Since the inequality $t_m^{v^{-1}(\beta^m)} \geq k_m -
s_m^{u^{-1}(\alpha)} +1$ is satisfied, such a $\beta$ must exist. Let
$$W_{\mathfrak{t}+1} = F_{\alpha} \cap G_{\beta}.$$ Set $$U_{k_m
-\mathfrak{t}-1} = U_{k_m-\mathfrak{t}} - \{\alpha\},
$$ 
$$V_{k_m-\mathfrak{t}- 1} = V_{k_m-\mathfrak{t} } - \{\beta\}$$ and
$M_{\mathfrak{t}+1} = M_{\mathfrak{t}} \cup \{ W_{\mathfrak{t}+1} \}$.
\medskip

\noindent{\bf Step 2.} The inductive loop in Step 1 terminates when $\mathfrak{t}={k_m}$.  The rank set associated to the Richardson variety $R(u,v)$ is $M(R(u,v))= M_{k_m}$.
\end{algorithm}

\medskip

\begin{remark}
Every vector space $W_i$ formed during the algorithm occurs as
$F_{\alpha} \cap G_{\beta}$. Hence, $W_i$ is the span of the
consecutive set of basis elements $e_{n-\beta + 1}, \dots,
e_{\alpha}$. Since each $\alpha$ and each $\beta$ occur only once
during the algorithm, $l(W_i) \not= l(W_j)$ and $r(W_i)\not= r(W_j)$
for $i\not=j$. Therefore, $M(R(u,v))$ is a rank set.
\end{remark}

Before proceeding to prove Theorem \ref{projection=rank}, we give some examples of Algorithm \ref{rich} and Algorithm \ref{rank}.

\begin{example}
Let $M$ be the rank set $W_1 = [e_1, e_7], W_2 = [e_2, e_6], W_3=[e_3,
e_4], W_4 = [e_4, e_5], W_5= [e_6, e_8]$ so $X(M)\subset G(5,8)$.
Then $(m_{1},\dotsc , m_{5})=(1,2,3,3,1)$ so the colors of the vector
spaces are $W_1^3, W_2^2, W_3^1, W_4^1, W_5^3$ as indicated by the
superscripts. Hence, the corresponding minimal Richardson variety is
contained in $Fl(2,3,5;8)$ and has defining permutations
$u=(4,5,6,7,8)$ and $v=(5,6,7,3,8)$.  \smallskip
 
Conversely, suppose we begin with the Richardson variety associated to
the permutations $u=(4,5,6,7,8)$ and $v=(5,6,7,3,8)$ in
$Fl(2,3,5;8)$.  Construct the following tables of associated data: 
\[
\begin{array}{c|c|c|c}
i &  u_{i} & c_{i} & I_{u}(i)\\
\hline 
1 &	4 &	1 &	(1,1,1)\\
2 &	5 &	1 &	(2,2,2)\\
3 &	6 &	2 &	(2,3,3)\\
4 &	7 &	3 &	(2,3,4)\\
5 &	8 &	3 &	(2,3,5)
\end{array}
\hspace{.5in}
\begin{array}{c|c|c|c}
i &  v_{i} & d_{i} & I_{v}(i)\\
\hline 
1 &	5 &	1 &	(1,1,2)\\
2 &	6 &	1 &	(2,2,3)\\
3 &	7 &	2 &	(2,3,4)\\
4 &	3 &	3 &	(0,0,1)\\
5 &	8 &	3 &	(2,3,5).
\end{array}
\]
In Algorithm \ref{rank}, the vector spaces that are formed are
$W_{1}=F_4 \cap G_6=[e_3,e_4]$, $W_{2}=F_5 \cap G_5 = [e_4, e_5]$,
$W_{3}=F_6 \cap G_7= [e_2, e_6]$, $W_{4}=F_7 \cap G_8 = [e_1, e_7]$
and $W_{5}=F_8 \cap G_3= [e_6, e_8]$.  We elaborate on the computation
of $W_{2}$.  For $\mathfrak{t}=2$, we have $\alpha = 5$ and $c=1$, so
$\beta^{1}=5, \beta^{2}=7, \beta^{3}=8$.  The minimum among the
$\beta^{d}$'s is 5 with $d=1$.  The tricky condition
\[
t_d^{v^{-1}(\beta^d)} \geq k_d - s_d^{u^{-1}(\alpha)} +1
\]
is satisfied since $1= t^{1}_{1}< k_{1} - s_{1}^{2} +1 =2 -2 +1 =1$,
hence $\beta =5$. Observe that we recover the initial rank set.
\end{example}

\begin{example}
Let $R(u,v)$ be the Richardson variety in $F(2,4;7)$ associated to the
permutations $u=(4,6,2,7)$ and let $v=(2,7,3,5)$. 
Construct the following tables of associated data: 
\[
\begin{array}{c|c|c|c}
i &  u_{i} & c_{i} & I_{u}(i)\\
\hline 
1 &	4 &	1 &	(1,2)\\
2 &	6 &	1 &	(2,3)\\
3 &	2 &	2 &	(0,1)\\
4 &	7 &	2 &	(2,4)
\end{array}
\hspace{.5in}
\begin{array}{c|c|c|c}
i &  v_{i} & d_{i} & I_{v}(i)\\
\hline 
1 &	2 &	1 &	(1,1)\\
2 &	7 &	1 &	(2,4)\\
3 &	3 &	2 &	(1,2)\\
4 &	5 &	2 &	(1,3).
\end{array}
\]
In Algorithm \ref{rank}, the vector spaces that are formed when computing $M(R(u,v))$ are 
$W_1=F_{2}\cap G_{7}=[e_1, e_2]$,
$W_2=F_{4} \cap G_{5}=[e_3, e_4]$, 
$W_3=F_{6} \cap G_{2}=[e_6]$, and  
$W_4=F_{7} \cap G_{3}=[e_5,e_7]$.

\smallskip

If we begin with the rank set $M=\{W_1=[e_1, e_2], W_2=[e_3, e_4],
W_3=[e_6], W_4=[e_5, e_7]\}$, then the colors of the vector spaces
assigned in the Algorithm \ref{rich} are $W_1^2, W_2^2, W_3^1,
W_4^2$. Hence, Algorithm \ref{rich} assigns the Richardson variety
$R(u',v')$ in $F(1,4;7)$, where $u' = (6,2,4,7)$ and
$v'=(2,3,5,7)$. Observe that $R(u',v')$ is different from $R(u,v)$. In
fact, they are not subvarieties of the same flag variety. The
projection of $R(u,v)$ to $X(M(R(u,v)))$ has positive dimensional
fibers, where as the projection map from $R(u',v')$ to $X(M(R(u,v)))$
is birational.
\end{example}

We are now ready to start the proof of Theorem \ref{projection=rank}. We first determine the dimension of rank varieties. 
 
 \begin{lemma}\label{dim}
 The rank variety $X(M)$ associated to a rank set $M$ is an irreducible subvariety of $G(k,n)$ of dimension $$\dim(X(M)) = \sum_{i=1}^k \dim(W_i) - \sum_{i=1}^k \#\{ W_j \in M \ | \ W_j \subseteq W_i \}.$$
 \end{lemma}
 \begin{proof}
 This is a special case of Lemma 3.29 in \cite{coskun:LR} and follows easily by induction on $k$. When $k=1$, the rank variety is projective space of dimension $\dim(W_1)-1$, as claimed in the lemma. Let $W_1$ be the vector space with minimal $l(W_i)$ in $M$. Omitting $W_1$  gives rise to a rank variety $X(M')$ in $G(k-1, n)$. There is  a dominant morphism from a dense open subset of $X(M)$ to $X(M')$ and the fibers are open subsets in a projective space of dimension $$\dim(W_1) - \# \{ W_i \in M \ | \ W_i \subseteq W_1\}.$$ The lemma follows by induction. 
 \end{proof}
 
 Next we show that the projection of $R(u,v)$ is $X(M(R(u,v)))$.
 
 \begin{lemma}\label{lemma-1}
 Let $R(u,v)$ be a Richardson variety in $F(k_1, \dots, k_m; n)$. Let $\pi$ denote the projection to $G(k_m, n)$. Then $\pi(R(u,v)) = X(M(R(u,v)))$. 
 \end{lemma}
 
 \begin{proof} Let $(V_1, \dots, V_m) \in R(u,v)$. We first prove that
$\pi(R(u,v)) \subset X(M(R(u,v)))$. It suffices to check that $V_m$
satisfies all the rank conditions imposed by $M(R(u,v))$. The basic
linear algebra fact is that $$\dim(V_r \cap F_{u_i} \cap
G_{v_j}) \geq s_r^i + t_r^j - k_r$$ since $V_r$ intersects
$F_{u_i}$ and $G_{v_j}$ in subspaces of dimension at least $s_r^i$ and
$t_r^j$, respectively. Furthermore, since $V_r \subset V_m$, we
conclude that $$\dim(V_m \cap F_{u_i} \cap
G_{v_j}) \geq \max_{1 \leq r \leq m} s_r^i + t_r^j -
k_r.$$ Now notice that $M(R(u,v))$ is constructed so that $\dim(V_m
\cap F_{u_i} \cap G_{v_j})$ precisely equals
$\max_{1 \leq r \leq m} s_r^i + t_r^j - k_r$. Hence, the rank
conditions imposed by $M(R(u,v))$ are satisfied by $V_m$ for $(V_1,
\dots, V_m) \in R(u,v)$.
 
Next, we show that the projection of $R(u,v)$ is onto $X(M(R(u,v)))$. Since $R(u,v)$ is a projective variety and $\pi$ is a morphism, the image $\pi(R(u,v))$ is a projective variety. Hence, it suffices to show that a general point of $X(M(R(u,v)))$ is in the image of $\pi$. There is a dense open set of $X(M(R(u,v)))$ consisting of $k$-planes $\Lambda$ such that $$\dim(\Lambda \cap W_i) = \# \{ W_j \in M \ | \ W_j \subseteq W_i \}$$ for every $i$ and $$\dim(\Lambda \cap W_i \cap W_j) = \# \{ W_t \in M \ | \ W_t  \subset W_i \cap W_j \}$$ for every $i,j$. Fix such a $k$-plane $\Lambda$ that has a basis $(b_1, \dots, b_k)$ with $b_i \in W_i$.
Let $\tau_s u$ be the truncation of the permutation $u$ obtained by taking the first $k_s$ numbers $(u_1, \dots, u_{k_s})$ in the permutation $u$. Similarly, let $\tau_s v$ be the corresponding truncation of $v$. 
For every $1 \leq s < m$, construct a sequence of vector spaces $W_1^s, \dots, W_{k_s}^s$ by running the Algorithm \ref{rank} with the permutations $\tau_s u$ and $\tau_s v$ for the flag variety $F(k_1, \dots, k_s; n)$. Inductively, we define a point of the Richardson variety $R(u,v)$ as follows. Let $V_m = \Lambda$. For every $W_i^{m-1}$ let $b_i^{m-1} = \sum_{b_j \in W_i^{m-1}} b_j$. Let $V_{m-1}$ be the span of the vectors $b_i^{m-1}$. Continuing by descending induction, let $b_i^s = \sum_{b_j^{s+1} \in W_i^s} b_j^{s+1}$. Let $V^s$ be the vector space spanned by the vectors $b_i^s$. In this way, we obtain a partial flag $(V_1, \dots, V_m)$. By construction, it is easy to see that this partial flag lies in both Schubert varieties $X_u$ and $X^v$, hence in the Richardson variety $R(u,v)$. Furthermore, $\pi((V_1, \dots, V_m)) = \Lambda$. We conclude that $\pi$ is surjective. This concludes the proof.
\end{proof}

 \begin{lemma}\label{lemma-2}
Let $M_0$ be a rank set for $G(k,n)$. Let $R(u,v)(M_0)$ be the associated minimal Richardson variety assigned by Algorithm \ref{rich}. Let $M(R(u,v))(M_0)$ be the rank set associated to $R(u,v)(M_0)$ by Algorithm \ref{rank}. Then $M(R(u,v)(M_0)) = M_0$.  
 \end{lemma} 
  
 \begin{proof}
 This is clear, hence left to the reader.
 \end{proof}

\begin{proof}[Proof of Theorem \ref{projection=rank}]
We are now ready to prove Theorem \ref{projection=rank}. By Lemma \ref{lemma-1}, every projection variety is a rank variety. By Lemma \ref{lemma-2}, every rank variety arises as a projection variety. 
These statements together imply that rank varieties are projection varieties and vice versa. 
\end{proof}

\begin{remark}\label{rem:t.fixed}
As a first application, we can determine the torus fixed points in the
smooth locus of a Richardson variety in $G(k,n)$. The rank set $M$
associated to a Richardson variety consists of $k$ vector spaces $W_1,
\dots, W_k$ such that they all have color one (equivalently, there are
no containment relations among different vector spaces $W_i$ and
$W_j$). We can assume that these vector spaces are ordered in
increasing order by $l(W_i)$. Equivalently, we can order the vector
spaces $W_i$ by $r(W_i)$ in increasing order. Since $W_i \not\subset
W_j$ for $i\not= j$, this leads to the same order. The torus fixed
points are $k$-dimensional subspaces that are spanned by $k$ distinct
basis elements $e_{i_1}, \dots, e_{i_k}$ with $i_1< i_2 < \cdots <
i_k$. For a particular torus fixed point to be contained in the rank
variety, we must have $e_{i_j} \in W_j$. To see this, note that $$(1)
\ \ \dim(Span(e_{i_1}, \dots, e_{i_k}) \cap Span(W_1, W_2, \dots,
W_j)) \geq j$$ and $$(2) \ \ \dim(Span(e_{i_1}, \dots, e_{i_k}) \cap
Span(W_k, W_{k-1}, \dots, W_j)) \geq k-j+1.$$ If $e_{i_j} \not\in
W_j$, then either $e_{i_1}, \dots, e_{i_j} \in [e_1, e_{l(W_j)-1}]$ or
$e_{i_j}, \dots, e_{i_k} \in [e_{r(W_j)+1}, e_n]$. The first case
contradicts the second inequality and the second case contradicts the
first inequality.  Now we are ready to characterize the torus fixed
points in the smooth locus of $X(M)$. They are spanned by $e_{i_1},
\dots, e_{i_k}$ with $i_1 < i_2 < \cdots < i_k$ and $e_{i_j} \in W_j$
such that:
\begin{enumerate}
\item If $l(W_{j+1}) > l(W_j) + 1$, then $e_{i_j} \not\in W_{j+1}$; and
\item If $r(W_{j-1}) < r(W_j) -1$, then $e_{i_j} \not\in W_{j-1}$
\end{enumerate}
To see that these are necessary and sufficient conditions, simply use Corollary \ref{richardson}  and the description of singularities of Schubert varieties in Grassmannians.
In particular, if $u_{i+1} > u_i + 1$ for $1 \leq i < k$ and $v_{i+1} > v_i +1$ for  $1 \leq i < k$, then the Richardson variety $R(u,v)$ has a torus fixed point in its smooth locus if and only if every vector space $W_j$ in the corresponding rank set contains a basis element $e_{i_j}$ which is not contained in any of the other vector spaces in the rank set. 
\end{remark}

\begin{theorem}\label{t:rank.variety}
Let $M$ be a rank set for $G(k,n)$. Let  $R(u,v)(M)$ be the Richardson variety associated to $M$ by Algorithm \ref{rich}. Let $$\pi: R(u,v)(M) \longrightarrow X(M)$$ be the corresponding projection morphism. Then $R(u,v)(M)$ is birational to $X(M)$ under $\pi$ and the exceptional locus of $\pi$ has codimension at least $2$.
\end{theorem}

\begin{proof}
Let $[\Lambda] \in X(M)$ be a $k$-dimensional subspace such that $$\dim(\Lambda \cap W_i) = \# \{ W_s \ | \ W_s \subseteq W_i\}$$ for every $i$ and $$\dim(\Lambda \cap W_i \cap W_j) = \# \{ W_s \in M \ | \  W_s \subset W_i \cap W_j \}$$ for every $i$ and $j$. The set of such $\Lambda$ form a dense, Zariski open subset $U$ of $X(M)$. To see that $U$ is not empty take a vector space $\Lambda$ spanned by vectors $\sum_{e_j \in W_i} \alpha_j^i e_j$, where the collection of coefficients $\alpha_j^i$ are algebraically independent. Then it is clear that $\Lambda$ is in $U$. The inverse of $\pi$ can be defined over $U$ as follows. Let $W_{i_s}$ be $1 \leq s \leq k_i - k_{i-1}$ be the vector spaces in $M$ that are assigned the color $i$.   Let $\Lambda_i $ be the span of the vector spaces $\Lambda \cap W_{i_s}$ with $ 1 \leq s \leq k_i - k_{i-1}$. Note that by construction $\Lambda_i$ is a subspace of $\Lambda$ of dimension $k_i$ containing $\Lambda_{i-1}$ and contained in $\Lambda_{i+1}$. It follows that the partial flag $(\Lambda_1, \dots, \Lambda_k = \Lambda)$ is the inverse image of $\Lambda$ under the projection map $\pi$. Hence $\pi$ is birational. 

We now bound the dimension of the exceptional locus. Note that the fiber dimension of $\pi$ is positive if and only if at least one of the vector spaces $V_j$ intersects $W_i^{c_i}$ with $c_i <j$ in a subspace of dimension greater than $\# \{W_s \in M \ | \ W_s \subseteq W_i\}$.  We can stratify the rank variety into loci where such intersections happen and compare the decrease in the dimension of the image of $\pi$ with the increase in the dimension of the fibers of $\pi$. In fact, by stratifying the rank variety successively, it suffices to carry out the calculation when $j = c+1$ and $c_i = c$.  Let $W_i$ and $W_j$ be two vector spaces with colors $c$ and $c+1$,  respectively, such that $W_i \cap W_j \not= \emptyset$. Define a new rank set $M'(W_i, W_j)$ as follows.

Step 1.  Let  $W_j' = W_j \cap W_i$.  List all the vector spaces $W_1, W_2, \dots, W_r$ of color $c+1$ in $M$ that contain $W_j'$ ordered so that $l(W_1) < l(W_2) < \cdots < l(W_r)$. Form a new {\em collection} of vector spaces $M'$ by replacing  $W_1, \dots, W_r$ in $M$ with $W_j'$ and the spans $\overline{W_1 W_2}, \overline{W_2, W_3}, \cdots, \overline{W_{r-1}W_r}$.  Note that $M'$ is not necessarily a rank set since two of the vector spaces may coincide or the least or largest index  basis elements in two of the vector spaces may coincide.

Step 2. Let $W-e_{r(W)}$ (respectively, $W-e_{l(W)}$) denote the vector space spanned by the set of all the basis elements in $W$ but $e_{r(W)}$ (respectively, $e_{l(W)}$). As long as there are two vector spaces $W_1 \subseteq W_2$ in $M'$ with $r(W_1) = r(W_2)$, replace $W_2$ in $M'$ with the vector space $W_2 - e_{r(W_2)}$ keeping its label the same and relabel the new collection of vector spaces $M'$. If there are no such vector spaces, as long as there are vector spaces $W_1 \subseteq W_2$ in $M'$ with $l(W_1) = l(W_2)$, replace $W_2$ in $M'$ with the vector space $W_2 - e_{l(W_2)}$ and relabel the new collection of vector spaces $M'$. The procedure terminates when $M'$ is a rank set or when one of the vector spaces consists only of the zero vector. In the former case, set $M'(W_i, W_j)= M'$. In the latter case, set $M'(W_i, W_j)= \emptyset$. We call this process the {\em normalization} of the set of vector spaces $M'$.

Observe that Step 2 leads to isomorphic subvarieties (see \cite{coskun:LR} for a discussion of the normalization algorithm). We include it in order to apply the dimension formula in Lemma \ref{dim} without modification. The fiber of $\pi$ over $M'(W_i, W_j)$ has dimension one. Note that the locus where $\pi$ has higher dimensional fibers can be obtained by repeated applications of Steps 1 and 2.  In order to estimate the dimension of the exceptional locus,  it suffices to compare the dimension of $X(M)$ to $X(M'(W_i, W_j))$. There are two cases to consider. If $W_i \subset W_j$, then $\dim(W_i) \leq \dim(W_j) -2$ since $W_i$ does not contain $l(W_j)$ and $r(W_j)$. Using the dimension formula given in  Lemma \ref{dim} and the fact that Step 2 can only decrease the value of the expression, we see that $\dim ( X(M'(W_i, W_j))) \leq \dim(X(M)) -r -2$. In particular, this dimension is at least three less. Hence, the exceptional locus has codimension at least two.

If $W_i \not\subset W_j$, then we may assume that $l(W_j) < l(W_i)$ and $r(W_j) < r(W_i)$. By the algorithm assigning colors, we know that there exists $W_t$ of color $c+1$ containing $W_i$ such that $ l(W_j) < l(W_t) < l(W_i)$. We conclude that $\dim(W_j \cap W_i) \leq \dim(W_j) -2$. By the dimension formula given in Lemma \ref{dim}, it follows that $\dim ( X(M'(W_i, W_j))) \leq \dim(X(M)) -r -2$. In particular, this dimension is at least three less. Hence, the exceptional locus has codimension at least two. 

This concludes the discussion that the exceptional locus has codimension at least two.
\end{proof}

The following corollary states a more precise version of Theorem \ref{main-2}.

\begin{corollary}\label{important}
Let $X(M)$ be a rank variety.  Let $\pi: R(u,v)(M) \rightarrow X(M)$ be
the projection from the minimal Richardson variety associated to $M$.
Then the singular locus of $X(M)$ is given by
$$X(M)^{sing} = \{ x \in X \ | \ \pi^{-1}(x) \in R(u,v)^{sing}\
\mbox{or} \ \dim(\pi^{-1}(x)) \geq 1 \}. $$ In particular, the
singular locus of $X(M)$ is a union of projection varieties.
\end{corollary}

\begin{remark}
More generally, the singular locus of any projection variety in an
arbitrary $G/P$ is a union of projection varieties.  However, it is
more complicated to determine the singular locus as above.  It is an
interesting open problem to find an explicit characterization in
general.
\end{remark}

The basic observation that allows us to characterize the singular loci of projection varieties is the following. 

\begin{lemma}
Let $f: X \rightarrow Y$ be a birational morphism of normal, projective varieties such that the exceptional locus $E$ of $f$ (i.e., the locus in $X$ where $f$ fails to be an isomorphism) has codimension at least two. Then $Y^{sing} = f(X^{sing} \cup E)$. 
\end{lemma}

\begin{proof}
The map $f$ gives an isomorphism between $X - E$ and $Y- f(E)$. Hence, $ (Y-f(E))^{sing} = f((X-E)^{sing})$. Consequently, the content of the lemma is that $f(E) \subset Y^{sing}$. It is well-known that $Y$ is badly singular along $f(E)$. For example, $Y$ cannot even be $\QQ$-factorial along $f(E)$. To see this, note that by Zariski's Main Theorem, the fibers of $f$ over the points of $f(E) \subset Y$ are positive dimensional. Since the question is local on $Y$, by replacing $Y$ by a Zariski open neighborhood containing $y$, we may assume that $f(E)=y$.  Let $C$ be a curve in the fiber of $f$ over $y$.  Let $D$ be a divisor on $X$ associated to a section of a very ample line bundle $A$. Since the exceptional locus of $f$ has codimension at least 2, $f(D)$ is a Weil-divisor on $Y$ containing $y$. Suppose $f(D)$ were $\QQ$-Cartier at $y$. Then $m f(D)$ would be the class of a line bundle $M$ for some $m>0$. Let $L$ be an ample line bundle on $Y$ such that $L \otimes M^{-1}$ is ample.  Then $f^*(L \otimes M^{-1})$ is the pull-back of an ample line bundle by a birational map, hence it is NEF.  In particular, it has non-negative degree on the curve $C$. However, the degree of  $f^*L$ on $C$ is zero and the degree of $f^*(M^{-1}) = A^{\otimes -m}$ is negative. We thus get a contradiction.  We conclude that $f(D)$ cannot be $\QQ$-Cartier at $y$. This concludes the proof of the lemma. 
\end{proof}

\noindent
\textit{Proof of Corollary~\ref{important}.}  
Consider the map $\pi:R(u,v) \rightarrow X(M)$. Since $X(M)$ is normal
and the map is birational, by Zariski's Main Theorem, $\pi$ is an
isomorphism over the locus $$U= \{x \in X(M)\ |
\dim(\pi^{-1}(x))=0\}.$$ Since the exceptional locus of $\pi$ has
codimension at least two, by the previous lemma, $X(M)$ is singular
along $X(M) - U $. Since over $U$ the map $\pi$ is an isomorphism,
$x\in U$ is singular if and only if $\pi^{-1}(x)$ is singular in
$R(u,v)$. This concludes the proof of the Corollary.
\qed

\bigskip

Note that Corollary \ref{richardson} and Corollary \ref{important} explicitly determine the irreducible components of the singular locus of projection varieties.

\begin{example}
We give a simple example showing how to find the singular locus of a projection variety. 
Let $M$ be the rank set $$W_1=[e_1, e_6], W_2= [e_3,e_4], W_3=[e_5, e_{10}], W_4=[e_7, e_8] .$$ Let $M_1$ be the rank set
$$W_1^1 = [e_3], W_2^1 = [e_4], W_3^1=[e_5, e_{10}], W_4^1=[e_7, e_8] .$$ Let $M_2$ be the rank set $$W_1^2=[e_1, e_6], W_2^2= [e_3,e_4], W_3^2= [e_7], W_4^2= [e_8].$$ Then $X(M_1)$ and $X(M_2)$ are the loci over which $\pi$ has positive dimensional fibers. 
Let $M_3$ be the rank set $$W_1^3= [e_1, e_5], W_2^3= [e_3,e_4], W_3^3=[e_5, e_6], W_4^3=[e_7, e_8].$$ Let $M_4$ be the rank set $$W_1^4= [e_3, e_4], W_2^4=[e_5,e_6], W_3^4=[e_6, e_{10}], w_4^4=[e_7, e_8].$$ By Corollary \ref{richardson}, $X(M_3)$ and $X(M_4)$ are the loci where $\pi^{-1}$ is singular. We conclude that $X(M)$ is singular along $\bigcup_{i=1}^4 X(M_i)$. 
\end{example}

Finally, we obtain a generalization of Corollary \ref{segre}.

\begin{definition}
Let $j \leq k$ and let $m \leq n-k+j$. Let $V$ be an $n$-dimensional
vector space and let $T$ and $U$ be $m$ and $(k-j)$-dimensional
subspaces of $V$ such that $T \cap U = \{0\}$.  A {\em linearly
embedded sub-Grassmannian} $G(j,m)$ in $G(k,n)$ is the image of $\phi:
G(j, T) \hookrightarrow G(k,V)$ under the map $\phi: W \mapsto W
\oplus U$. A {\em Segre product of linearly embedded
sub-Grassmannians} is a product of linearly embedded sub-Grassmannians followed by the Segre embedding 
$$G(j_1, m_1) \times \cdots G(j_r, m_r) \hookrightarrow G(k_1, n_1) \times \cdots \times G(k_r, n_r) \hookrightarrow G(\sum k_i, n).$$
as described in Definition
\ref{definition-segre}.   
\end{definition}
\smallskip

\begin{remark}
A linearly embedded sub-Grassmannian $G(j,m) \subset G(k,n)$ is a
smooth Schubert variety with class $\sigma_{(n-k)^{k-j},
(n-m+k-j)^j}$.  In fact, every smooth Schubert variety in $G(k,n)$ is
a linearly embedded sub-Grassmannian.
\end{remark}

\begin{corollary}
Let $X$ be a rank variety with rank set $M$.  The following are equivalent.  
\begin{enumerate}
\item $X$ is smooth. 
\item $X$ is a Segre product of linearly embedded sub-Grassmannians. 
\item $M$ is a union of 1-dimensional subspaces and rank sets on
disjoint intervals which correspond with sub-Grassmannians after
quotienting out by the 1-dimensional subspaces.
\end{enumerate}

\end{corollary}

For example, $G(2,4)$ is the smooth rank variety with rank set
$\{[e_{1},e_{3}], [e_{2}, e_{4}] \}$.  In $G(7, 12)$, the
rank set
\[
M=\{[e_{1}, e_{5}], [e_{2},e_{6}], [e_{3},e_{7}], [e_{4}], [e_{8},e_{10}], [e_{9},e_{11}],[e_{12}] \}
\]
corresponds with a smooth rank variety isomorphic to the product of $G(3,6) \times G(2,4)$.

\begin{proof}
Let $X$ be a Segre product of linearly embedded sub-Grassmannians $$G(j_1, m_1) \times \cdots \times G(j_s, m_s) \hookrightarrow G(k_1, n_1) \times \cdots \times G(k_s, n_s) \hookrightarrow G(k,n).$$ Then $X$ is smooth. We need to show that we can realize $X$ as a rank variety. The Segre product of rank varieties $X(M_i)$ is a rank variety corresponding to the concatenation of the corresponding rank set. A Schubert variety in $G(k_i, n_i)$ is a rank variety since it is a Richardson variety in the Grassmannian. Since a linearly embedded sub-Grassmannian is a Schubert variety, we conclude that a Segre product of linearly embedded sub-Grassmannians is a smooth rank variety. 

Conversely, suppose that $X$ is a smooth rank variety. We show that $X$ has to be a Segre product of linearly embedded sub-Grassmannians. Consider the corresponding rank set $M$. If $M$ contains only one vector space, then $X(M)$ is projective space and the corollary holds. If the color of all the vector spaces in the rank set is one, then $X(M)$ is a Richardson variety in the Grassmannian and the corollary holds by Corollary \ref{segre}. Now we will do induction on the number of vector spaces defining $M$. We may assume that there are some vector spaces in $M$ assigned a color larger than one. Let $W$ be a subspace in the rank set that is assigned the color $1$. If $\dim(W)=1$, then we can replace every vector space $W_i$ in the rank set $M$ by $W_i/W$. We obtain a new rank set $M'$ for $G(k-1, n-1)$ with one fewer vector space. The map $f: X(M') \rightarrow X(M)$ sending $\Lambda \in X(M')$ to the span of $\Lambda$ and $ W$ is an isomorphism between $X(M')$ and $X(M)$. By induction, $X(M')$ is a Segre product of linearly embedded sub-Grassmannians. It follows that $X(M)$ is a Segre product of linearly embedded Grassmannians.  If $\dim(W)>1$, we show that $X(M)$ is singular. Take a vector space $W'$ of color two containing $W$. Define a new set of vector spaces $M'$ by replacing $W'$ and $W$ with the two vector spaces $[e_{l(W)}, e_{r(W)-1}], [e_{l(W)+1}, e_{r(W)}]$. If $M'$ is a rank set, stop. The fiber of $\pi$ over this locus is positive dimensional. Hence, $X(M)$ is singular. If $M'$ is not a rank set, normalize the set of vectors to obtain a rank set $M''$. Note that $M''$ is non-empty and the fiber of $\pi$ over $X(M'')$ is positive dimensional. Hence, $X(M)$ is singular. This concludes the proof. 

The equivalence of (2) and (3) follows from the fact that $G(k,n)$ is itself a
rank variety corresponding with $M=\{W_{1},\ldots , W_{k} \}$ where each $W_{i}=[e_{i},e_{n-k+i}]$.   

\end{proof}

There is a nice way to enumerate all the rank varieties in $G(k,n)$
using the Stirling numbers of the second kind.  In fact, if we
$q$-count the rank varieties according to dimension, we get a well
known $q$-analog of the Stirling numbers \cite{ER96,GR86,Milne78,WW91}. 

Define the generating function 
\[
g[k,n] = \sum_{M} q^{\dim(X(M))}
\]
where the sum is over all rank sets $M$ for $G(k,n)$.  Set $g[k,n]=0$ for
$k>n$, $g[0,n]=0$ for $n>0$ and $g[0,0]=1$.  Let $[k]=1+q+\dotsb +
q^{k}$.  

\begin{lemma}\label{l:stirling}
The polynomials $g[k,n]$ satisfy the recurrence
\[
g[k,n] = g[k,n-1]+ [n-k+1]\cdot g[k-1,n-1]
\]
for $1\leq k\leq n$.
\end{lemma}

\begin{proof}
Every rank set $M$ in $G(k,n)$ is either a rank set in $G(k,n-1)$ or
it contains a subspace of the form $[e_{i}, e_{n}]$.  In the
latter case removing this subspace leaves a rank set $M'$ in
$G[k-1,n-1]$ which does not include a subspace whose left endpoint is
$i$.  Observe that $\dim(X(M))- \dim(X(M'))$ equals $n-i$ minus the
number of subspaces in $M'$ with left endpoint larger than $i$.
Furthermore, for each $0\leq d \leq n-k$, we can add a subspace with
right endpoint $n$ to $M'$ to get a rank set in $G(k,n)$ of dimension
$d+\dim(X(M))$ by choosing the left endpoint to be the $d$-th largest
value in $\{1,2,\dotsc , n \} - \{r(W): W \in M' \}$.
\end{proof}

Recall that the Stirling numbers of the second kind $S(n,k)$ count the
number of set partitions of $\{1,\dotsc , n \}$ into $k$ nonempty
blocks.  Let $S[n,k]$ be the $q$-analog of $S(n,k)$ defined by the
recurrence
\[
S[n,k] = q^{k-1}  S[n-1, k-1 ]  + [k] S[n-1,k]
\] 
with initial conditions $S[0,0]=1, S[n,0] =0$ for $n>0$, and
$S[n,k]=0$ for $k>n$. One can show that $S[n,k]$ is divisible by
$q^{\chs{k}{2}}$.  Then, simple algebraic manipulations prove the
following corollary to Lemma~\ref{l:stirling}.

\begin{corollary}\label{c:stirling}
For $1\leq k\leq n$, we have $g[k,n] = S[n+1,n-k+1] \cdot q^{-\chs{n-k+1}{2}}$.  
\end{corollary}

Below are the 25 rank sets for $G(2,4)$ listed by dimension.  Thus
$g[2,4] = 6 + 8q + 7q^{2} +3q^{3}+ q^{4})$.  Here $(34,123)$ means the
rank set consisting of two subspaces spanned by $<e_{3},e_{4}>$ and
$<e_{1},e_{2},e_{3}>$.

\[
\begin{array}{rl}
dim &  rank sets\\
0: &  (2,1) , (3,1) , (4,1) , (3,2) , (4,2) , (4,3)\\
1: &  (2 3,1) , (3 4,1) , (3,1 2) , (4,1 2) , (2,1 2 3) , (3 4,2), (4,2 3) , (3,2 3 4)\\
2: &  (2 3 4,1) , (2 3,1 2) , (3 4,1 2) , (4,1 2 3) , (2,1 2 3 4), (3,1 2 3 4) , (3 4,2 3)\\
3: &  (2 3 4,1 2) , (3 4,1 2 3) , (2 3,1 2 3 4)\\
4: &  (2 3 4,1 2 3)
\end{array}
\]


\begin{thebibliography}{10}

\bibitem{BGG}
{\sc I.~Bernstein, I.~Gelfand, and S.~Gelfand}, {\em {Schubert cells and
  cohomology of the spaces $G/P$}}, Russian Math. Surveys, 28 (1973),
  pp.~1--26.

\bibitem{B4}
{\sc S.~Billey}, {\em {Kostant polynomials and the cohomology ring for $G/B$}},
  Duke Math. J., 96 (1999), pp.~205--224.

\bibitem{BLak}
{\sc S.~Billey and V.~Lakshmibai}, {\em {Singular Loci of Schubert Varieties}},
  no.~182 in Progress in Mathematics, {Birkh\" auser}, 2000.

\bibitem{BW-sing}
{\sc S.~C. Billey and G.~S. Warrington}, {\em Maximal singular loci of
  {S}chubert varieties in {${\rm SL}(n)/B$}}, Trans. Amer. Math. Soc., 355
  (2003), pp.~3915--3945 (electronic).

\bibitem{Brion:flag}
{\sc M.~Brion}, {\em Lectures on the geometry of flag varieties}, in Topics in
  cohomological studies of algebraic varieties, Trends Math., Birkh\"auser,
  Basel, 2005, pp.~33--85.

\bibitem{BGY}
{\sc K.~A. Brown, K.~R. Goodearl, and M.~Yakimov}, {\em Poisson structures on
  affine spaces and flag varieties. {I}. {M}atrix affine {P}oisson space}, Adv.
  Math., 206 (2006), pp.~567--629.

\bibitem{CK}
{\sc J.~B. Carrell and J.~Kuttler}, {\em On the smooth points of t-stable
  varieties in g/b and the {P}eterson map}, preprint,  (1999).

\bibitem{cortez}
{\sc A.~Cortez}, {\em Singularit\'es g\'en\'eriques et quasi-r\'esolutions des
  vari\'et\'es de {S}chubert pour le groupe lin\'eaire}, Adv. Math., 178
  (2003).

\bibitem{coskun:DP}
{\sc I.~Coskun}, {\em The enumerative geometry of {Del Pezzo} surfaces via
  degenerations}, American Journal of Mathematics, 128 (2006), pp.~751--786.

\bibitem{coskun:LR}
\leavevmode\vrule height 2pt depth -1.6pt width 23pt, {\em A
  {L}ittlewood-{R}ichardson rule for two-step flag varieties}, Invent. Math.,
  176 (2009), pp.~325--395.

\bibitem{Dem}
{\sc M.~Demazure}, {\em {D\'esingularization des Vari\'et\'es de Schubert
  G\'en\'eralis\'ees}}, Ann. Sc. E.N.S., 4 (1974), pp.~53--58.

\bibitem{duan}
{\sc H.~Duan}, {\em Multiplicative rule of {S}chubert classes}, Invent. Math.,
  159 (2005), pp.~407--436.

\bibitem{duan-zhao}
{\sc H.~Duan and X.~Zhao}, {\em Erratum: {M}ultiplicative rule of {S}chubert
  classes}, Invent. Math., 177 (2009), pp.~683--684.

\bibitem{ER96}
{\sc R.~Ehrenborg and M.~Readdy}, {\em Juggling and applications to
  {$q$}-analogues}, in Proceedings of the 6th {C}onference on {F}ormal {P}ower
  {S}eries and {A}lgebraic {C}ombinatorics ({N}ew {B}runswick, {NJ}, 1994),
  vol.~157, 1996, pp.~107--125.

\bibitem{Fulton}
{\sc W.~Fulton}, {\em {Intersection Theory}}, Springer-Verlag, New York, 1984.

\bibitem{Fulton-book}
{\sc W.~Fulton}, {\em Young Tableaux; With Applications To Representation
  Theory And Geometry}, vol.~35 of London Mathematical Society Student Texts,
  Cambridge University Press, New York, 1997.

\bibitem{fulton-harris}
{\sc W.~Fulton and J.~Harris}, {\em Representation Theory}, vol.~129 of
  Graduate Texts in Mathematics, Springer-Verlag, New York, 1991.
\newblock A first course, Readings in Mathematics.

\bibitem{GR86}
{\sc A.~M. Garsia and J.~B. Remmel}, {\em {$Q$}-counting rook configurations
  and a formula of {F}robenius}, J. Combin. Theory Ser. A, 41 (1986),
  pp.~246--275.

\bibitem{GL-book}
{\sc N.~Gonciulea and V.Lakshmibai}, {\em Flag Varieties}, Hermann-Acutalities
  Mathematiques, 2001.

\bibitem{harris}
{\sc J.~Harris}, {\em Algebraic geometry}, vol.~133 of Graduate Texts in
  Mathematics, Springer-Verlag, New York, 1995.
\newblock A first course, Corrected reprint of the 1992 original.

\bibitem{hartshorne}
{\sc R.~Hartshorne}, {\em Algebraic Geometry}, Springer-Verlag, New York, 1977.

\bibitem{Hum-LAG}
{\sc J.~E. Humphreys}, {\em Linear Algebraic Groups}, vol.~21 of Graduate texts
  in mathematics, Springer-Verlag, New York, 1975.

\bibitem{klr}
{\sc C.~Kassel, A.~Lascoux, and C.~Reutenauer}, {\em The singular locus of a
  {S}chubert variety}, J. Algebra, 269 (2003), pp.~74--108.

\bibitem{kleiman}
{\sc S.~L. Kleiman}, {\em The transversality of a general translate},
  Compositio Math., 28 (1974), pp.~287--297.

\bibitem{KLS}
{\sc A.~{Knutson}, T.~{Lam}, and D.~E. {Speyer}}, {\em {Positroid varieties I:
  juggling and geometry}}, ArXiv e-prints,  (2009).

\bibitem{kollar}
{\sc J.~Koll{\'a}r}, {\em Higher direct images of dualizing sheaves. {I}}, Ann.
  of Math. (2), 123 (1986), pp.~11--42.

\bibitem{kollar:mori}
{\sc J.~Koll{\'a}r and S.~Mori}, {\em Birational geometry of algebraic
  varieties}, vol.~134 of Cambridge Tracts in Mathematics, Cambridge University
  Press, 1998.
\newblock With the collaboration of C. H. Clemens and A. Corti, Translated from
  the 1998 Japanese original.

\bibitem{KK}
{\sc B.~Kostant and S.~Kumar}, {\em {The Nil Hecke Ring and Cohomology of $G/P$
  for a Kac-Moody Group $G^*$}}, Advances in Math., 62 (1986), pp.~187--237.

\bibitem{kovacs}
{\sc S.~J. Kov{\'a}cs}, {\em A characterization of rational singularities},
  Duke Math. J., 102 (2000), pp.~187--191.

\bibitem{kreiman}
{\sc V.~Kreiman and V.~Lakshmibai}, {\em Richardson varieties in the
  {G}rassmannian}, in Contributions to automorphic forms, geometry, and number
  theory, Johns Hopkins Univ. Press, Baltimore, MD, 2004, pp.~573--597.

\bibitem{kumar}
{\sc S.~Kumar}, {\em The nil {Hecke} ring and singularity of {Schubert}
  varieties}, Inventiones Math., 123 (1996), pp.~471--506.

\bibitem{kumar-book}
\leavevmode\vrule height 2pt depth -1.6pt width 23pt, {\em Kac-Moody Groups,
  Their Flag Varieties and Representation Theory}, vol.~204 of Progress in
  Mathematics, Birkhauser, 2002.

\bibitem{Lusztig94}
{\sc G.~Lusztig}, {\em Total positivity in reductive groups}, in Lie theory and
  geometry, vol.~123 of Progr. Math., Birkh\"auser Boston, Boston, MA, 1994,
  pp.~531--568.

\bibitem{manivel}
{\sc L.~Manivel}, {\em Le lieu singulier des vari\'et\'es de {S}chubert},
  Internat. Math. Res. Notices,  (2001), pp.~849--871.

\bibitem{manivel-book}
{\sc L.~Manivel}, {\em Symmetric Functions, {Schubert} Polynomials and
  Degeneracy Loci}, vol.~6 of SMF/AMS Texts and Monographs, American
  Mathematical Society, 2001.

\bibitem{Milne78}
{\sc S.~C. Milne}, {\em A {$q$}-analog of restricted growth functions,
  {D}obinski's equality, and {C}harlier polynomials}, Trans. Amer. Math. Soc.,
  245 (1978), pp.~89--118.

\bibitem{polo94}
{\sc P.~Polo}, {\em On {Z}ariski tangent spaces of {S}chubert varieties, and a
  proof of a conjecture of {D}eodhar}, Indag. Math. (N.S.), 5 (1994),
  pp.~483--493.

\bibitem{postnikov-2006}
{\sc A.~Postnikov}, {\em Total positivity, grassmannians, and networks}, 2006.

\bibitem{purbhoo}
{\sc K.~Purbhoo}, {\em Vanishing and nonvanishing criteria in {S}chubert
  calculus}, Int. Math. Res. Not.,  (2006).

\bibitem{richardson}
{\sc R.~W. Richardson}, {\em Intersections of double cosets in algebraic
  groups}, Indag. Math. (N.S.), 3 (1992), pp.~69--77.

\bibitem{rietsch-2006}
{\sc K.~Rietsch}, {\em Closure relations for totally nonnegative cells in
  {$G/P$}}, Math. Res. Lett., 13 (2006), pp.~775--786.

\bibitem{WW91}
{\sc M.~Wachs and D.~White}, {\em {$p,q$}-{S}tirling numbers and set partition
  statistics}, J. Combin. Theory Ser. A, 56 (1991), pp.~27--46.

\bibitem{wiki:flags}
{\sc Wikipedia}, {\em Generalized flag variety --- wikipedia{,} the free
  encyclopedia}, 2010.
\newblock [Online; accessed 14-June-2010].

\end{thebibliography}

\def\cprime{$'$}

\end{document}